\documentclass{amsart}
\usepackage{amsmath, amssymb}  
\usepackage{bbm}
\usepackage{enumerate}
\usepackage[all]{xy}

\usepackage{hyperref}
\usepackage{fullpage}

\numberwithin{equation}{section}

\newtheorem{thm}{Theorem}
\newtheorem{pro}[thm]{Proposition}
\newtheorem{lm}[thm]{Lemma}
\newtheorem{cor}[thm]{Corollary}
\theoremstyle{definition}
\newtheorem{df}[thm]{Definition}
\newtheorem{remark}[thm]{Remark}
\newtheorem{hyp}[thm]{Hypothesis}

\newtheorem{notn}[thm]{Notation}

\newcommand{\doubleslash}{/\negthinspace/}

\def\G{\mathrm{G}}
\def\H{\mathrm{H}}
\def\X{\mathrm{X}}
\def\mm{\mathrm{m}}

\def\b{\mathfrak{b}}
\def\g{\mathfrak{g}}
\def\n{\mathfrak{n}}
\def\q{\mathfrak{q}}
\def\z{\mathfrak{z}}
\def\h{\mathfrak{h}}
\def\t{\mathfrak{t}}
\def\s{\mathfrak{s}}

\def\Gal{\mathrm{Gal}}
\def\N{\mathrm{N}}
\def\Z{\mathrm{Z}}
\def\GL{\mathrm{GL}}
\def\SL{\mathrm{SL}}
\def\Sp{\mathrm{Sp}}
\def\Sym{\mathrm{Sym}}

\def\sl{\mathfrak{sl}}

\DeclareMathOperator{\spec}{Spec}

\DeclareMathOperator{\Hom}{Hom}
\DeclareMathOperator{\rk}{rk}
\def\B{\mathrm{B}}

\let\Sec\S
\def\S{\text{S}}
\def\T{\mathrm{T}}
\DeclareMathOperator*{\Int}{Int}
\DeclareMathOperator*{\Ad}{Ad}
\def\QQ{\mathbb{Q}}
\def\CC{\mathbb{C}}
\def\NN{\mathbb{N}}
\def\RR{\mathbb{R}}
\def\ZZ{\mathbb{Z}}
\def\GG{\mathbb{G}}
\def\mcA{\mathcal{A}}
\def\mcB{\mathcal{B}}

\def\mcL{\mathcal{L}}
\def\mcY{\mathcal{Y}}
\DeclareMathOperator{\Res}{Res}

\def\tn{\mathrm{tn}}
\def\der{\mathrm{der}}
\def\scn{\mathrm{sc}}
\DeclareMathOperator{\ad}{ad}
\def\mO{\mathfrak O}
\DeclareMathOperator*{\meas}{meas}
\def\charfn{\mathbbm{1}}
\DeclareMathOperator*{\Lie}{Lie}
\newcommand{\dsum}{\displaystyle\sum}
\newcommand{\Fbar}{\overline{F}}
\newcommand{\Xs}{X_{\mathrm{s}}}

\newcommand{\Gm}{\GG_{\mathrm{m}}}

\def\ggood/{$\fg$-good}
\def\gFgood/{$\fg$-$F$-good}
\def\gEgood/{$\fg$-$E$-good}
\def\ngood/{$\n^-$-good} 
\def\esstame/{not too wild}
\def\blahblah/{condition (T)}

\newcommand{\set}[2]{
\left\{
\left.
#1\vphantom{#2\bigl(\bigr)}\,
\right|
\,#2
\right\}
}

\def\bZ{\mathbb{Z}}
\def\fg{\mathfrak{g}}
\def\fh{\mathfrak{h}}
\newcommand{\ov}{\overline}

\usepackage{comment}

\begin{document}
\title{On Kostant Sections and Topological Nilpotence}
\author{Jeffrey D. Adler}
\address{
Department of Mathematics and Statistics \\
American University \\
4400 Massachusetts Ave NW \\
Washington, DC  20016-8050 \\
USA 
}
\email{jadler@american.edu}
\author{Jessica Fintzen}
\address{
Department of Mathematics\\
University of Michigan \\
2074 East Hall\\
530 Church Street\\
Ann Arbor, MI 48109 \\
USA
}
\email{fintzen@umich.edu}
\author{Sandeep Varma}
\address{
School of Mathematics \\
Tata Institute of Fundamental Research \\
Homi Bhabha Road \\
Colaba, Mumbai  \\
India
}
\email{sandeepvarmav@gmail.com}

\thanks{%
Jessica Fintzen's research was partially supported by the Studienstiftung des deutschen Volkes and a postdoctoral fellowship of the German Academic Exchange Service (DAAD)}

\date{January 31, 2018}

\begin{abstract}
Let $\G$ denote a connected, quasi-split reductive group over
a field $F$ that is complete with respect to a discrete valuation
and that has a perfect residue field.
Under mild hypotheses, we
produce a subset of the Lie algebra $\g(F)$ that picks out
a $\G(F)$-conjugacy class in every stable, regular, 
topologically nilpotent conjugacy class in $\g(F)$. This generalizes an earlier result obtained by DeBacker and one of the authors under stronger hypotheses.
We then show that
if $F$ is $p$-adic,
then the characteristic
function of this set behaves well with respect to endoscopic transfer.
\end{abstract}
\maketitle

\section{Introduction}

Let $\G$ be a (quasi-split) connected reductive group over a field $F$,
with Lie algebra $\g$.
In \cite{Kos63}, assuming $F$ to be an algebraically closed
field of characteristic zero,
B. Kostant gave a simple, elegant
and remarkably useful recipe to construct sections (now called
\emph{Kostant sections}) to the geometric invariant theory (GIT) quotient
$\g \rightarrow \g \doubleslash \G$,
taking values in the set of regular elements
of $\g$. When $\G$ is a
general linear group, the sections thus constructed include
as a special case the companion matrix
found in elementary linear algebra. In this paper, taking $F$ to
be a $p$-adic field (or more generally
a complete discrete valuation field), we make the case that constructing such sections
integrally and studying them help us better understand a certain subset of
$\g(F)$ that has shown up in certain problems related
to harmonic analysis on $p$-adic groups.
Based on these results we  exhibit new examples of pairs of functions
that match each other
in the sense of endoscopic transfer for Lie algebras.

Let $F$ be a complete discrete valuation field,
and let $p\geq 0$ denote the characteristic of the residue field
$\kappa$,
which we assume to be perfect.
Let $\G$ be
a quasi-split, connected, reductive group over $F$,
and $Y \in \g(F) = \Lie \G(F)$
a regular nilpotent element. The first result
that we prove in this paper is the following statement
that sharpens and extends the main result of \cite{AD04}.
Under mild hypotheses (see later in this introduction),
we present a neighborhood of $Y$ --- let us call it $Y + \g_{x, 0+}$
in view of notation that will be established later ---
with the following two properties:
\begin{enumerate}[(a)]
\item
if $X \in \g(F)$, then the $\Ad \G(\Fbar)$-orbit
of $X$ intersects $Y + \g_{x, 0+}$ if and only if
$X$ is regular and topologically nilpotent; and
\item
if $X \in \g(F)$ is regular and topologically nilpotent, then
the
intersection $\Ad \G(\Fbar)(X) \cap \left( Y + \g_{x, 0+} \right)$
is a single $\Ad \G_{x, 0+}$-orbit,
for a fixed (independent of $X$) bounded open subgroup
$\G_{x, 0+} \subset \G(F)$.
\end{enumerate}
As the notation suggests,
in (a) and (b) above, $x$ stands for a certain point in the
Bruhat--Tits building of $\G$, and $\g_{x, 0+}$ (resp., $\G_{x, 0+}$)
is the associated Moy--Prasad lattice (resp., subgroup).

The characteristic function
$\phi = \charfn_{Y + \g_{x, 0+}}$ of the set
$Y + \g_{x, 0+}$, like its several variants mentioned below, has
been no stranger to harmonic analysis, partly
because of the role played by nilpotent elements in representation
theory, and partly because orbital integrals, whose evaluation on
$\phi$ is facilitated by (b) above, are important to harmonic analysis.
For instance, when
$\G = \GL_n$ over a $p$-adic field $F$, J. Repka
used the function $\phi$
(or rather, its composite with
$g \mapsto g - 1$) in order to compute Shalika
germs associated to the regular unipotent conjugacy class
of $\G(F)$ (see \cite{Rep84}).
More general
variants of this set have shown up in many important
works on the subject, especially in the context of character
theory and nilpotent orbits --- for instance the set
$n_{\alpha} + L_j$ in \cite[the proof of Lemma 6]{How74},
its generalization $Y + \varpi^n {L'}^{\perp}$ in 
\cite[the proof of Proposition I.11]{MW87}, and the sets
$X + \g_{F^*}^+$ in \cite{Deb02b}.

As alluded to earlier, S. DeBacker and the first-named author proved
(see \cite[Proposition 1]{AD04}) the aforementioned
result on the properties (a) and (b) of the set $Y + \g_{x, 0+}$
under more restrictive hypotheses, and only for regular
semisimple elements $X$.
The hypotheses of \cite{AD04}
require that certain
reductive groups over finite fields admit suitably
well behaved $\sl_2$-triples.
For example,
if $\G = \Sp_{2n}$,
then the use of 
\cite[Hypothesis 4.2.3]{Deb02a}
requires $p$ to be at least $4n + 1$ if nonzero.
Our result, by contrast, is always valid for $p>5$
if $\G$ is semisimple, tamely ramified
and has no factor of type $A_n$.
Moreover, if $\G = \GL_n$, then we impose no restriction on $p$.
We hope that our new presentation,
in addition to weakening hypotheses,
makes certain aspects of the role of the Kostant
section more explicit.

The second result says that,
if $F$ is a finite extension of $\QQ_p$ and
$p$ satisfies a few further hypotheses,
then the function $\phi$
behaves well with respect to endoscopic
transfer.
In other words,
suppose that the conditions of Hypothesis \ref{hyp:Sec3} are
satisfied by $\G$ as well as by a group $\H$ that is endoscopic for $\G$,
and that
$\phi_{\H} = \charfn_{Y_{\H} + \h_{x_{\H}, 0+}}
\in C_c^{\infty}(\h(F))$ is the function obtained by applying the
construction of $\phi$ to $\H$ in place of $\G$.
The statement then is that, up to an
explicitly computable nonzero scalar, $\phi$ and
$\phi_{\H}$ have matching orbital integrals.
This fact can be used 
to cook up more pairs of functions with matching
orbital integrals.
The interest in such results comes from the fact that the theory
of endoscopy uses orbital integrals to relate harmonic analysis
on $\H(F)$ with that on $\G(F)$, but the supply of explicit pairs of functions
with matching orbital integrals in the literature is somewhat limited
(for an example of some deep work on this question, see the paper
\cite{KV12} by D. Kazhdan and Y. Varshavsky).

Now let us remark on some considerations
that motivated our proof of (a) and (b) above.
The proof of \cite[Proposition 1]{AD04}
makes use of a hypothesis that $Y$ can be completed to an $\sl_2$-triple
containing another nilpotent element $Y'$ such that
the following equation holds (as also its analogues
over finite extensions of $F$)
\[
Y + \g_{x, 0+} = \Ad \G_{x, 0+} \left(Y + C_{\g_{x, 0+}}(Y') \right),
\]
where $C_{\g_{x, 0+}}(Y')$ is the centralizer of $Y'$ in $\g_{x, 0+}$.
Note that $Y + C_{\g_{x, 0+}}(Y')$ is part of the Kostant section
$Y + C_{\g(F)}(Y')$ (cf.\ \cite{Kos63}) attached
to $Y$ and $Y'$. Thus, the assertion of \cite{AD04} is that
$Y + C_{\g_{x, 0+}}(Y')$ is precisely the set of topologically nilpotent
elements in the Kostant section $Y + C_{\g(F)}(Y')$.

Now suppose $\G$ is unramified.
Then $x$ is hyperspecial and gives
a realization of $\G$ as a reductive group
over the ring $\mO$ of integers of $F$.
In this case
(for $p$ satisfying the hypotheses of \cite{Deb02a}),
$Y + C_{\g_{x, 0}}(Y')$ is a Kostant section for $\G$ \emph{over $\mO$},
from which the above claim about $Y + C_{\g_{x, 0+}}(Y')$ follows
easily.
Moreover, as alluded to above,
one may make this argument work for $p$
satisfying much weaker hypotheses,
at least provided we replace $C_{\g_{x, 0}}(Y')$ by a different
$\mO$-submodule of $\g_{x, 0}$
(compare with Remark 19${}^\prime$ of \cite{Kos63}).

However, since the main result of \cite{AD04} is valid
even when $\G$ is not unramified, one might wish for a proof that
still explicitly incorporates the above idea and yet works for
ramified groups, at least under mild hypotheses.
This is what we do here in \Sec\ref{AD04 result}.

\textbf{A few remarks on the hypotheses.}
Several results that we prove in
\Sec\ref{AD04 result} and \Sec\ref{pairs of matching functions}
of this paper do not hold for all reductive groups $\G$
and all residue field characteristics $p$.
In order to nevertheless state our results in large generality,
we will need to use the following conditions
(see
Definitions \ref{df:good}, \ref{df:esstame}, and \ref{df:blah},
and Hypothesis \ref{hyp:Sec3}):
\begin{itemize}
\item that of $p$ being \ggood/ --- namely, a `good prime' for $\G$ in the well known sense
of \cite{Spr66};
\item that of $p$ being \ngood/ --- a certain condition weaker than \ggood/;
\item that of $p$ being \gFgood/ --- a variation on `\ggood/' adapted to the graded
Lie algebra associated to a suitable Moy--Prasad filtration on $\g(F)$;
\item that of $\G$ being \esstame/ --- roughly speaking, tameness conditions
on the absolutely simple groups that $\G$ is built from, and on the `interaction'
between $\T$ and $\T_{\der} := \T \cap \G_{\der}$, where $\T$ is
a maximal torus of $\G$ and $\G_{\der}$ is the derived group of $\G$;
\item that of $\G$ satisfying \blahblah/ --- i.e., a maximal torus of $\G$ becomes 
an induced torus after passing to a tamely ramified extension of $F$ (this ensures
the existence of mock exponential maps);
\item that of two groups $\G$ and $\H$ satisfying Hypothesis
\ref{hyp:Sec3} --- a set of conditions designed to ensure a few extra conveniences
such as the existence of a Kazhdan--Varshavsky quasi-logarithm.
\end{itemize}

Note that things simplify considerably
when $\G$ is split over a tamely ramified extension of $F$. In this case,
$\G$ is automatically \esstame/ and satisfies
\blahblah/. Further, in this case, $p$ is \gFgood/ if and only if it is
\ggood/.

Suppose $\G$ is defined and split over $\mO$.
Our construction of a candidate for being named a `Kostant' section to
the adjoint quotient map $\g \rightarrow \g \doubleslash \G$ over $\mO$ already
requires $p$ to be \ngood/, and this condition turns out to suffice
to guarantee the existence of such sections. The assertion (a) stated early
on in this introduction requires $\G$ to be \esstame/ and $p$ to be \ngood/,
while assertion (b) requires that $p$ is in fact \gFgood/
and that $\G$ also satisfy \blahblah/. As stated
earlier, Hypothesis \ref{hyp:Sec3} shows up and is required only in
\Sec\ref{pairs of matching functions}, for our results involving
endoscopy.

\textbf{Structure of the paper.}
After setting up basic notation (\Sec\ref{basic notation})
and recalling a definition of topological nilpotence
(\Sec\ref{tn definition}),
we construct in \Sec\ref{Kostant section split}
an integral version of the Kostant section when $\G$ is split,
under hypotheses somewhat milder
than those in \cite{Ric16} (see Hypothesis \ref{hyp for Kostant section}).
In fact, our Kostant section is built using a certain $\mO$-module, and we see that whenever this module exists,
we do indeed obtain a Kostant section (Proposition \ref{Kostant section pro}).
The main objects that concern us
($\G$, $x$, $Y$, etc.) are set up in \Sec\ref{set up}, using constructions
recalled in \Sec\ref{pinnings etc}.
In \Sec\ref{assumptions on p}, we define what it means for the residual
characteristic $p$ of $F$ to be
``\ggood/'',
``\gFgood/'',
or
``\ngood/''
for $\G$.
These conditions assure that $\ad Y$ has good regularity properties 
over the residue field $\kappa$ of $F$.

In \Sec\ref{Kostant sections over F OE}, we first consider the case where $\G$
is split over a tamely ramified extension $E$ of $F$,
and use an integral Kostant section
$Y+L_{\mO_E}$
for the base change
of $\G$ to $E$, together with
Lemma \ref{Kostant section over F},
to get a Kostant
section $Y + L_F$ for $\G$ (over $F$), whose subset of topologically
nilpotent elements is precisely a set of the form $Y + (L_F \cap \g_{x, 0+})$
in a sense alluded to above.
In the same section, under the assumption that
$p$ is ``\gFgood/'', we extend this last assertion to $\G$ satisfying
the weaker condition of being ``\esstame/''.
In fact,
this last condition is necessary in order to have an $F$-subspace $L_F$
that is compatible with the $\mO_E$-module $L_{\mO_E}$.
The Kostant section supplied by \Sec\ref{Kostant sections over F OE}
is related to $Y + \g_{x, 0+}$ in
\Sec\ref{relating with Kostant section} under the condition
that $p$ is \gFgood/ for $\G$, when $\G$ satisfies
\blahblah/.
Our generalization of the main result of \cite{AD04} is then deduced
in \Sec\ref{with AD04 main result}.

The assertion about endoscopic transfer is proved in \Sec\ref{main result 2}.
Since the
notion of topological nilpotence behaves well with respect to the
matching of conjugacy classes in the theory of endoscopy, there
are two main issues to take care of: that of the transfer factors
and that of the normalized orbital integrals.
It turns out that Theorem 5.1 of
\cite{Kot99}, recalled along with a review of endoscopic
transfer for Lie algebras in \Sec\ref{adjoint quotients map},
immediately tells us how to handle the former.
While the property (b) 
of the set $Y + \g_{x, 0+}$ discussed near the beginning of this
introduction makes it believable
that the orbital integrals can be easily handled, one needs to
do a bit more work since different elements in this set generally have
nonisomorphic centralizers.
However, it turns out
that the $\Ad \G(F)$-orbits of elements of $Y + \g_{x, 0+}$
have measures arising from Kirillov's symplectic forms on them,
which are easily evaluated on the intersection of $Y + \g_{x, 0+}$
with these orbits, thanks to $Y$ and $\g_{x, 0+}$ being conveniently adapted
to these forms (see \Sec\ref{studying omegaX'},
particularly Lemma \ref{selfdual}).
Executing this requires passage to the Lie algebra,
to which end we impose a stronger hypothesis on
$p$ (Hypothesis \ref{hyp:Sec3})
so as to make use of a Kazhdan--Varshavsky quasi-logarithm; see
\Sec\ref{KV consequences}.
Further, in \Sec\ref{omegas relation}, specifically Lemma
\ref{omegaX omegaX'} below, we relate the measures arising
from Kirillov's symplectic form construction with a different set
of choices
for these measures that is fixed in \Sec\ref{comments on measures},
the latter being better suited for 
studying endoscopic transfer.
This lets us finish the computation of the orbital integrals in
\Sec\ref{orbital integral computation}.
In fact, it is our comparison between measures in Lemma \ref{omegaX omegaX'}
that accounts for the `$\Delta_{\text{IV}}$' transfer factors,
or equivalently, the normalization of the orbital integrals.
Finally, \Sec\ref{more pairs of matching functions} discusses how
standard techniques allow us to cook up more pairs of matching functions
starting from $(\phi, \phi_{\H})$ as above.

\textbf{Acknowledgements.}
The third-named author thanks Radhika Ganapathy,
for it was collaborative work with her that brought many
of the points here to his attention;
and Dipendra Prasad,
for helpful pointers and continued encouragement and support.
The authors thank Stephen DeBacker for useful conversations,
and the referee for helpful suggestions.
For hospitality during the writing of this paper,
the authors thank
the Indian Institute for Science Education and Research (Pune),
the Institute for Mathematical Sciences (National University of Singapore),
and
the Centre International de Rencontres Math\'ematiques (Luminy).

\section{Topologically nilpotent elements in a Kostant Section}
\label{AD04 result}
\subsection{Basic notation} \label{basic notation}
Let $F$ be a complete, discretely valued field with perfect
residue field,
$\Fbar$
a fixed algebraic closure of $F$.
For any extension $E$ of
$F$ in $\Fbar$, $\mO_E$ will denote the ring of integers of $E$, and,
if $E/F$ has finite ramification degree,
$\varpi_E$ will denote a uniformizer in $\mO_E$ and $\kappa_E$ the residue
field $\mO_E/\varpi_E \mO_E$.
Let $p\geq 0$ denote the characteristic of $\kappa_F$
and $\ov \kappa_F$ an algebraic closure of $\kappa_F$.
Let $|\cdot|$ stand for an absolute value on $F$,
extended uniquely to $\Fbar$.
Throughout, $\G$ will be a connected reductive group over a ring
which, most of the time
(everywhere but in \Sec\ref{Kostant section split}), will equal $F$.
We will let $\Z^0$ denote the identity connected component of the
center of $\G$,
$\G_{\der}$ the derived group of $\G$,
$\G_{\scn}$ the simply connected cover of $\G_{\der}$,
and
$\G_{\ad}$ the adjoint group of $\G$. Thus, we have obvious maps
$\G_{\scn} \rightarrow \G \rightarrow \G_{\ad}$. However, we will
make the following exception to this convention: if $\T \subset
\G$ is a maximal torus, $\T_{\der}, \T_{\scn}$ and $\T_{\ad}$ will
denote the maximal tori of $\G_{\der}, \G_{\scn}$ and $\G_{\ad}$
determined by $\T$.
We will follow standard notation in denoting algebraic groups
using upper case roman letters and their Lie algebras using the corresponding
fraktur letters, e.g., $\g_{\der}$ is the Lie algebra of $\G_{\der}$.

For any extension $E/F$ in $\Fbar$ of finite ramification degree,
let $\mcB(\G,E)$ denote the reduced Bruhat--Tits building of $\G$ over $E$.
If $E/F$ is Galois, then we have a canonical injection
$\mcB(\G,F) \longrightarrow \mcB(\G,E)^{\Gal(E/F)}$.
For $x \in \mcB(\G, E)$ and $r \in \RR$, write
$\g(E)_{x, r} \subset \g(E)$ and (when $r \geq 0$)
$\G(E)_{x, r} \subset \G(E)$ for the corresponding Moy--Prasad lattice
and subgroup, respectively.

We will omit $E$ from all of the notation above when $E=F$,
e.g., $\varpi$ will mean $\varpi_F$.

If $\X$ is a scheme over a ring $R$, and $R'$ is an $R$-algebra,
then $\X_{R'}$ will usually
denote the base change of $\X$ to $R'$.

If $\T\subseteq \G$ is a maximal torus, then $R(\G,\T)$ will denote
the set of (absolute) roots of $\G$ with respect to $\T$.
If $\B$ is a Borel subgroup of $\G$ containing $\T$, then
$R(\B,\T) \subset R(\G,\T)$ will denote the corresponding
set of positive roots, and $\Delta(\B,\T)$ the corresponding
set of simple roots.

\subsection{Topological nilpotence} \label{tn definition}

In this section we take $\G$ to be a connected reductive group over $F$.

\begin{df} \label{top nilp}
Let $X \in \g(F)$ and let $\T$ be an 
$\Fbar$-torus in $\G$ such that $\t(\Fbar)$ 
contains
the semisimple part $\Xs$ of $X$.
We say that $X$
is \emph{topologically nilpotent} if
\[
\text{$|d\mu(\Xs)| < 1$
for all $\mu \in \X^*(\T) := \Hom_{\text{alg}}(\T, \Gm)$}.
\]
\end{df}

\begin{remark}
\begin{enumerate}[(a)]
\item
It is easy to see that this definition is independent of the choice
of a torus $\T$ whose Lie algebra contains $\Xs$.
\item
This definition is one of several that are commonly used.
In
Remark \ref{reconciling topological nilpotence},
we will see that it is equivalent to the one 
given in \cite{AD04},
once one has assumed that $p$ is \emph{\ggood/ for $\G$}
and that $\G$ is \emph{\esstame/},
concepts that will be introduced
later below.
\end{enumerate}
\end{remark}

\begin{notn}
Write $\g(F)_{\tn}$ for the set of topologically nilpotent elements of
$\g(F)$.
\end{notn}

We now give
a description of topological nilpotence using the adjoint quotient,
by which we mean the geometric invariant theory (GIT) quotient $\g \doubleslash \G$
corresponding to the adjoint action of $\G$ on $\g$.

\begin{notn} \label{adjoint quotient notation}
	Write $\chi : \g \longrightarrow \q := \g\doubleslash\G$ for the
	adjoint quotient map of $\g$ over $F$. In future sections, we will
	assume this notation with $F$ replaced by any ring $R$ over which
	$\G$ is defined.
\end{notn}

Note that any pinning of $\G_E$, where $E$ is a finite extension of
$F$ that splits $\G$, determines an $\mO_E$-model $\G_{\mO_E}$ for the base change
$\G_E$ of $\G$ to $E$
(see \Sec\ref{pinnings etc} and \Sec\ref{set up} below).
Thus, $\q_E$ gets the
integral model $\spec \mO_E[\g_{\mO_E}]^{\G_{\mO_E}}$, which is independent
of the choice of the $E$-pinning that defined $\G_{\mO_E}$ (as these
pinnings are all $\G_{\ad}(E)$-conjugate). Thus,
$\q_E$ gets a canonical $\mO_E$-model for any finite extension $E$ of
$F$ that splits $\G$, and if two such extensions $E_1$ and
$E_2$ are contained in another such extension $E_3$, then
the base changes of $\q_{\mO_{E_1}}$ and $\q_{\mO_{E_2}}$ to
$\mO_{E_3}$ agree as they both equal $\q_{\mO_{E_3}}$ (\cite[Lemma 2]{Ses77}). Thus, now
we may talk of $\q(\mO_{\Fbar})$, as well as of
$\q(\kappa_{\Fbar}) = \q(\bar \kappa)$.

\begin{lm} \label{unramified tn}
Let $\q(1)$ denote the fiber over $\chi(0)$ of
the reduction map $\q(\mO_{\Fbar}) \longrightarrow \q(\bar \kappa)$.
We have
\[
\g(F)_{\tn} = \chi|_{\g(F)}^{-1}(\q(1)).
\]
\end{lm}
\begin{proof}
Intersecting with $\g(F)$ yields this equality from its
analogue over a finite extension $E$ of $F$,
so we may and do assume $\G$ to be defined and split over $\mO$.

Let us first prove that, for any $X \in \t(\mO_{\Fbar})$, where $\T$
is an $\mO$-split maximal torus of $\G$, the following are equivalent:
\begin{enumerate}[(i)]
\item The image of $X$ in $\t(\bar \kappa)$ vanishes; and
\item $\chi(X)$ belongs to the inverse image of $\chi(0)$ under $\q(\mO_{\Fbar})
\rightarrow \q(\bar \kappa)$.
\end{enumerate} 
It is easy to see that (i) implies (ii), so it is enough to show that if the image
$\bar X$ of $X$ in $\t(\bar \kappa)$ does not vanish, then there exists a non-constant homogeneous polynomial $f
\in \mO[\g]^{\G}$ such that $\overline{f(X)} \neq 0$.
Choose any faithful representation $\iota: \G \hookrightarrow \GL(L)$ of group schemes over
$\mO$, where $L$ is a finite free $\mO$-module.
The coefficients of
the characteristic polynomial map on $\text{End}_{\mO}(L)$
determine homogeneous, non-constant elements
$f_1, \dots, f_n \in \mO[\g]^{\G}$,
where $n$ is the $\mO$-rank of $L$ (we are omitting
the coefficient of the highest degree term).
If $\bar X $ is nonzero,
then $\iota(\bar X)$ is a nonzero semisimple element of $\GL(L)(\ov \kappa)$, and hence some coefficient of its characteristic polynomial does not vanish.
But this means that $\overline{f_i(X)} \neq 0$ for some $i$, proving the claim.

Now let $X \in \g(F)$, and let
$X_s'$ be any conjugate in $\t(\Fbar)$ of the semisimple
part of $X$. By definition, $X$ is topologically nilpotent if and
only if $|d\mu(X_s')| < 1$ for all $\mu \in \X^*(\T)$, i.e., if and only
if $X_s'$ both belongs to $\t(\mO_{\Fbar})$ and has
zero as its image in $\t(\kappa_{\Fbar}) = \t(\bar \kappa)$. If $X_s' \not\in \t(\mO_{\Fbar})$, then we can multiply
$X_s'$ by some non-unit $a \in \mO_{\Fbar}$ such that $a X_s'$ belongs to
$\t(\mO_{\Fbar})$ and has nonzero image in $\t(\bar \kappa)$.
By the proof of the equivalence of (i) and (ii) above applied to $a X_s'$ in place of $X$,
we see that there exists a non-constant homogeneous polynomial $f \in \mO_{\Fbar}[\g]^G$ such that the image of $f(aX_s') \in \mO_{\Fbar}$ in $\bar \kappa$ is nonzero, i.e. $f(aX_s')$ is not divisible by $a$. This implies that $f(X_s') \notin \mO_{\Fbar}$, and hence $\chi(X)=\chi(X_s') \not \in \q(\mO)$, in particular $X \notin \chi|_{\g(F)}^{-1}(\q(1))$. Now the equality $\g(F)_{\tn} = \chi|_{\g(F)}^{-1}(\q(1))$  follows from the equivalence of (i) and (ii) above applied to $X_s'\in \t(\mO_{\Fbar})$ in place of $X$.
\end{proof}

\subsection{Kostant sections for split groups} \label{Kostant section split}
 
In this section,
we will take $\G$ to be a connected split reductive group defined
over $\ZZ$, but will concern ourselves with its base change $\G_R$
to a ring $R$ (commutative, with unity).
Fix a pinning $(\B, \T, \{X_{\alpha}\}_{\alpha \in \Delta})$,
for $\G$ that is defined over $\ZZ$, with $\Delta = \Delta(\B, \T)$.

\begin{notn} \label{new lambda weights}
Let $Y = \sum_{\alpha \in \Delta} X_{\alpha}$.
Write
\begin{equation} \label{new lambda weight spaces}
\g = \bigoplus_{j \in \ZZ} \g(j)
\end{equation} 
for the weight space decomposition for $\g$, over $\ZZ$, with
respect to $\Ad \circ \lambda$, where $\lambda \in \X_*(\T)$ is
the sum of {\em all} the coroots of $\T$ in $\B$ (not just
the simple ones). Let $\B^-$ be the
Borel subgroup of $\G$ that is opposite to $\B$ with respect to $\T$,
with $\N^-$ as its unipotent radical.
\end{notn}

It is easy to see
that $\langle \alpha, \lambda \rangle = 2$
for all $\alpha \in \Delta$, so that $\g(0) = \t$ and
\begin{equation} \label{bb-n- weights}
\b = \bigoplus_{j \geq 0} \g(j), \ \ \b^- = \bigoplus_{j \leq 0} \g(j),
\ \ \n^- = \bigoplus_{j < 0} \g(j).
\end{equation}
Further, 
\begin{equation} \label{ad Y on weight spaces}
\text{for all $j \in \ZZ$}, \quad [Y, \g(j)] \subset \g(j + 2).
\end{equation}

For the rest of \Sec\ref{Kostant section split}
we assume the following statement.

\begin{hyp} \label{hyp for Kostant section}
The module $[Y, \n^-_R]$ has a $\lambda$-invariant complement in $\b^-_R$
of rank $\rk_R \b^-_R - \rk_R \n^-_R$.
\end{hyp}
Call such a complement $\Xi$, and set $\mathcal S = Y + \Xi$.

The hypothesis is equivalent to requiring the following:
for each $j < 0$, the $R$-submodule
$[Y, \g_R(j)]$ (i.e., $[Y, \g(j)_R]$)
has a complement in $\g_R(j + 2)$
that is free of rank $\rk_R \g_R(j + 2) - \rk_R \g_R(j)$.

Note that $[Y, \g_R]$ need not have any complement in $\g_R$
even when the hypothesis is satisfied.
For example, consider $\G=\SL_p$ over $R=\ZZ_p$.

\begin{remark} \label{exclusions in hyp for Kostant section}
We can rephrase Hypothesis \ref{hyp for Kostant section} as follows.
Let $N$ be the square-free natural number whose prime factors
are precisely the ones occurring in the following list:
\begin{itemize}
\item $2$ if $\G_{\scn}$ has a factor of the form
$B_n$ ($n \geq 3$), $D_n$ ($n \geq 4$),
or $G_2$;
\item $2$ and $3$  whenever $\G_{\scn}$ has a factor of the form
$F_4$, $E_6$, or $E_7$; and
\item $2$, $3$, and $5$ whenever $\G_{\scn}$ has a factor
of the form $E_8$;
\item the primes dividing the order of $\pi_1(\G_{\der})$.
\end{itemize}
Note that the last condition is superfluous
unless $\G$ has a factor of the form $C_n$ or $A_n$.

By (2.4), (2.5), and (2.6) of \cite{Spr66},
Hypothesis \ref{hyp for Kostant section} is equivalent to $N$ being
invertible in $R$ (for this amounts, in the
notation \emph{op.\ cit.}, to saying that the elementary divisors
of $t_i$ are all invertible in $R$ for $i < 0$).
\end{remark}

Thus, for instance,
if $\G = \Sp_{2n}$ ($n \geq 2$),
$\SL_n$, or $\GL_n$, there is no restriction on $R$.

\begin{notn}
In the rest of this subsection we will abuse notation by confusing
$\G$, $\b$, $\n^-$, etc.,
with their base changes to $R$.
\end{notn}

Note that $\N^-$ acts by the adjoint action on $Y + \b^-$.
We have a composite
map $\mathcal S \hookrightarrow \g \rightarrow \g\doubleslash\G$,
which factors through the composite map
$\mathcal S \hookrightarrow Y + \b^- \rightarrow (Y + \b^-)\doubleslash\N^-$.

In this section we prove,
thanks to Hypothesis \ref{hyp for Kostant section}, that:
\begin{pro} \label{Kostant section pro}
Both maps in the sequence $\mathcal S \rightarrow (Y + \b^-)\doubleslash\N^-
\rightarrow \g\doubleslash\G$ are isomorphisms of schemes over $R$.
\end{pro}

\begin{remark} \label{what does Kostant section pro say}
Proposition \ref{Kostant section pro}
says that $\mathcal S$ can be viewed as a Kostant section for $\G$ over $R$.
Thus,
we have obtained a version of the Kostant section in a situation more general
than that in \cite{Ric16}. What makes this feasible is that, unlike in
\cite{Ric16}, we do not need that the Chevalley morphism
$\t \doubleslash W \rightarrow \g \doubleslash \G$ is an isomorphism
(which it is not under our less restrictive hypothesis,
by Theorem 1.2 of \cite{CR10}).
\end{remark}

The proof of the proposition is based on the following two lemmas.
\begin{lm} \label{a is injective}
The action map
$a: \N^- \times \mathcal S \rightarrow Y + \b^-$ is an isomorphism.
\end{lm}
\begin{proof}
The proof is adapted from that of Proposition 3.2.1 of \cite{Ric16}, which was
in turn adapted from \cite[Lemma 2.1]{GG02}.
We need to show that the map at the level
of coordinate rings $a^*: R[Y + \b^-] \rightarrow
R[\N^- \times \mathcal S]$ is an isomorphism. 

We first prove that $a^*$ is injective.
By viewing $\N^-$ as the product of the schemes underlying its
root subgroupschemes, we can view the coordinate
rings of either side as
$\Sym^{\bullet}({\n^-}^{\vee} \oplus \Xi^{\vee})$
and $\Sym^{\bullet}({\b^-}^{\vee})$
(symmetric powers of free $R$-modules).
The map $a$ is thus a polynomial map in several
variables in $R$, and its linear term,
as a map $\n^- \oplus \Xi \rightarrow \b^-$, equals:
\[
a_1: (X, Z) \mapsto [X, Y] + Z \ \ (X \in \n^-, Z \in \Xi).
\]
This linear term is an isomorphism of affine spaces, thanks to
Hypothesis \ref{hyp for Kostant section}.  
It now suffices to show that any polynomial self-map of an affine space
over $R$ whose degree one term is the identity is dominant. But this
follows since the induced map at the level of coordinate rings then necessarily
preserves the terms of the lowest degree. 

Now let us prove that $a^*$ is surjective.
The map $a$ is equivariant for the following
actions of $\Gm$: it acts by $\Int \circ \lambda$ on $\N^-$
and by $t \cdot x = t^{-2} \Ad \lambda(t) (x)$ on $\mathcal S$ and
$Y + \b^-$. These actions of $\Gm$ define decompositions of
$R[\N^- \times \mathcal S]$ and $R[Y + \b^-]$ into eigenspaces.

We first show that for each $\chi \in \X^*(\Gm) = \ZZ$, the
$\chi$-eigenspaces in $R[\N^- \times \mathcal S]$ and
$R[Y + \b^-]$ are free $R$-modules of the same, finite, rank. Since 
the relevant coordinate rings are $\Gm$-equivariantly identified
with $\Sym^{\bullet}({\n^-}^{\vee} \oplus \Xi^{\vee})$
and $\Sym^{\bullet}({\b^-}^{\vee})$
(where the action on $\n^-$ is via $\Ad \lambda(t)$
and on $\Xi$ and $\b^-$ via $t^{-2} \Ad \lambda(t)$),
the assertion about eigenspaces being free modules of
finite rank follows from the fact that the $\chi$-eigenspaces
in the degree-one parts of either ring are zero for $\chi \geq 0$
(the $\Gm$-actions contract the affine spaces to our base points,
as $t \rightarrow \infty$).
Now, to prove that the ranks of the $\chi$-eigenspaces match, it is enough
to do so degree by degree (i.e., to show that for each $n$,
the $\chi$-eigenspaces of $\Sym^n({\n^-}^{\vee} \oplus \Xi^{\vee})$
and $\Sym^n({\b^-}^{\vee})$ have the same rank).
This reduces to the degree one situation, which follows from
$a_1$ being a $\Gm$-equivariant isomorphism of tangent spaces.

Now, for each $\chi \in \X^*(\Gm)$, the restriction $a^*_{\chi}$
of $a^*$ to the $\chi$-eigenspace of $R[Y + \b^-]$ is an $R$-linear map between
two free $R$-modules of the same finite rank. Thus, choosing
bases for these $\chi$-eigenspaces, it is enough to show that the square
matrix that represents this restriction has determinant that is
a unit, i.e., one that survives reduction modulo any maximal ideal of
$R$. Since $a^*$ respects base change, we are now reduced to assuming
that $R$ is a field. But in this case, $a^*_{\chi}$, being an injective
map between $R$-vector spaces of the same dimension, is also a surjection.
\end{proof}

The following lemma does not use Hypothesis \ref{hyp for Kostant section},
and hence would be valid even if $R$ were arbitrary.
\begin{lm} \label{Y+b- dominant}
The map $R[\g]^{\G} \rightarrow R[Y + \b^-]^{\N^-}$ is injective.
\end{lm}
\begin{proof}
{\em Step 1:} We first show that the restriction map $R[\g]^{\G} \rightarrow
R[\b^-]$ is injective. For this consider the conjugation
map $\G \times \b^- \rightarrow \g$. This map is $\G$-equivariant, where
we let $\G$ act on $\G \times \b^-$ by
left translation along the first factor, and on $\g$ by the adjoint action.
This map is defined over $\ZZ$, and it suffices to show that 
it is universally schematically dominant relative to $\ZZ$ (for this would give
an injection $R[\g] \rightarrow R[\G] \otimes_R R[\b^-]$
that by $\G$-equivariance maps
$R[\g]^{\G}$ to $R[\G]^{\G} \otimes_R R[\b^-] = R[\b^-]$).
Since $\bZ$ is noetherian and $\G\times \b^-$ is flat, by Theorem 11.10.9 of \cite{Gro66},
it suffices to prove that the morphism $\G \times \b^- \rightarrow \g$
of $R$-schemes is schematically dominant when $R = k$ is a field,
which we may assume to be algebraically closed.
In this case, the map $\G \times \b^- \rightarrow \g$ is
surjective at the level of $k$-points, by Proposition 14.25 of \cite{Bor91},
giving the desired dominance as $\g$ is reduced.

{\em Step 2:} Now we show that restriction map $R[\g]^{\G} \rightarrow
R[\b^-]$ factors through the map $R[\t] \rightarrow R[\b^-]$ that is dual to
the $\T$-equivariant projection $\b^- \rightarrow \t$. Thus, we need to show
that the map $R[\g]^{\G} \rightarrow R[\b^-]$ equals the composite
$R[\g]^{\G} \rightarrow R[\t] \rightarrow R[\b^-]$ (the restriction to $\t$,
and then the $\T$-equivariant projection). In other words, if $\varphi \in
R[\g]^{\G}$, $R'$ is an $R$-algebra and $X = X_0 + X_- \in \b^-(R')$ with
$X_0 \in \t(R')$ and $X_- \in \n^-(R')$, then we need to show that
$\varphi(X) = \varphi(X_0)$.

Since the weights of $\lambda^{-1}$ on $\b^-$ are all positive,
there exists a unique morphism $\iota$
from the affine line over $R'$ to $\g_{R'}$, which on $\spec R'[t, t^{-1}]$
is given by $t \mapsto \Int \lambda^{-1}(t)(X_0 + X_-)$, and takes $0$ to $X_0$.
Since $\varphi$ is $\G$-invariant, $\varphi \circ
\iota$ is a constant
on $\spec R'[t, t^{-1}]$, and hence on the entire affine line over $R'$.
This gives
\[ \varphi(X_0) = \varphi \circ \iota(0) =
\varphi \circ \iota(1) = \varphi(X), \]
as needed.

{\em Step 3:} By Steps 1 and 2, the restriction map $R[\g]^{\G} \rightarrow
R[\t]$ is injective. Applying Step 2, this time with $\b$ in place of $\b^-$,
the restriction maps $R[\g]^{\G} \rightarrow R[\t]$ and
$R[\g]^{\G} \rightarrow R[Y + \t]$ coincide, once the `translation by $Y$'
identification of $\t$ with $Y + \t$ is made. Thus, the restriction
map $R[\g]^{\G} \rightarrow R[Y + \t]$ is injective, and
{\em a fortiori}, the map $R[\g]^{\G} \rightarrow R[Y + \b^-]$ is injective.
\end{proof}

\begin{remark}
The proof of Lemma \ref{Y+b- dominant}
yields a shorter proof of the result proved in
\cite{CR10} that the Chevalley morphism
$\t \doubleslash W \rightarrow \g \doubleslash \G$
over an arbitrary scheme is dominant whenever $\G$ is
almost simple (a restriction we see is no longer necessary).
\end{remark}

We now prove Proposition \ref{Kostant section pro}.
\begin{proof}[of Proposition \ref{Kostant section pro}]
Since $R[\mathcal S]$ is free over $R$ and the map
\[ R[Y + \b^-] \rightarrow R[\N^- \times \mathcal S] =
R[\N^-] \otimes_R R[\mathcal S] \]
is $\N^-$-equivariant and an isomorphism (Lemma \ref{a is injective}),
it follows that the map
$R[Y + \b^-]^{\N^-} \rightarrow R[\mathcal S]$ is an isomorphism too.
Therefore, it is enough to show that the map
$R[\g]^{\G} \rightarrow R[Y + \b^-]^{\N^-}$, which we already
know to be an injection, is also surjective. This assertion is
independent of $\mathcal S$, so we may choose any $\mathcal S = Y + \Xi$
we like. We choose $\Xi$ to be of the form $\dot \Xi_R$
where $\dot \Xi \subset \b^-(\ZZ[1/N])$ is a $\lambda$-invariant
$\ZZ[1/N]$-submodule complementary to $[Y, \n^-(\ZZ[1/N])]$.
It suffices to show that the composite map
$R[\g]^{\G} \hookrightarrow R[Y + \b^-]^{\N^-} \rightarrow
R[\mathcal S]$ is surjective.
Considering the chain
\[
\ZZ[1/N][\g]^{\G} \otimes_{\ZZ[1/N]} R \rightarrow
R[\g]^{\G} \hookrightarrow R[Y + \b^-]^{\N^-} \rightarrow
R[\mathcal S] = \ZZ[1/N][\mathcal S] \otimes_{\ZZ[1/N]} R,
\]
it is enough to prove the assertion of the proposition for $R = \ZZ[1/N]$.
We now follow
the proof of Corollary 3.7 of \cite{CR10}.
$\CC[\g]^{\G} \rightarrow \CC[\mathcal S]$ is an isomorphism, and hence
(using faithful flatness and \cite[Lemma 2]{Ses77})
so is $\QQ[\g]^{\G} \rightarrow \QQ[\mathcal S]$. Given
$P \in \ZZ[1/N][\mathcal S]$, one can therefore find $Q \in
\QQ[\g]^{\G}$ with image $P$. Write $Q = (r/s) Q_0$, where $Q_0$
is a primitive polynomial in $\ZZ[1/N][\g]$, and $r$ and $s$ are coprime
in $\ZZ[1/N]$. Now if some prime $p$ that does not divide $N$
divides $s$, then the image of $r Q_0$ in $(\ZZ/p\ZZ)[\mathcal S]$
is zero. But this means, by the injectivity of
$(\ZZ/p\ZZ)[\g]^{\G} \hookrightarrow (\ZZ/p\ZZ)[\mathcal S]$, that
$r Q_0 = 0$ in $(\ZZ/p\ZZ)[\g]$, a contradiction.
\end{proof}

\subsection{Pinnings, regular nilpotent elements, and hyperspecial points}
\label{pinnings etc}
From now on, we assume that $\G$ is a quasi-split reductive group over $F$.
Let $E$ be a finite Galois extension of $F$ in $\Fbar$ that splits
some (hence any) maximal $F$-torus in a Borel $F$-subgroup of $\G$.
Our goal in this subsection is to recall the following maps,
both of which are equivariant under the actions of $\G(F)$:
\begin{equation} \label{pinning orbit model}
\begin{aligned}
\begin{xy}
\xymatrix@R=0pt{
&
\{
	\text{regular nilpotent elements of $\g(F)$}
\}
\\
\{
	\text{$F$-pinnings of $\G$} 
\} \qquad
\ar[ur]!L
\ar[dr]!L
&
\\
&
\{
	\text{$E$-hyperspecial points in $\mcB(\G)$}
\}
}
\end{xy}
\end{aligned}
\end{equation}
Here, by an $F$-pinning
$(\B, \T, \{X_{\alpha}\}_{\alpha \in \Delta})$ of $\G$, we mean that
$(\B,\T)$ is a Borel-torus pair for $\G$, defined
over $F$,
$\Delta = \Delta(\B,\T)$,
each $X_\alpha$ is a nonzero vector in $\g_\alpha(E)$, where
$\g_\alpha$ is the $\alpha$-eigenspace for the action of $\T_E$ on $\g_E$,
and the set $\{X_\alpha\}$ is stable under $\Gal(E/F)$.
Note that we automatically have $X_{\alpha}
\in \g(E')$ for any extension $E'$ of $F$ in $\Fbar$ that splits $\G$ (and hence
$\T$). Thus the definition of an $F$-pinning is independent of the choice of $E$.

Then the upper arrow in \eqref{pinning orbit model} 
sends an $F$-pinning $(\B, \T, \{X_{\alpha}\})$ to the element
$\sum_{\alpha \in \Delta} X_{\alpha}$ of $\g(F)$, which
is regular nilpotent (see Lemma 3.1.1 of
\cite{Ric16}). This map is clearly $\G_{\ad}(F)$-equivariant.

When $F$ has characteristic zero, it is well known that this map induces
	a bijection at the level of $\G(F)$-conjugacy classes
	(cf.\ \cite[\Sec5.1]{LS87}). Even though we do not use it, let us mention that similar methods work in general when $p$ is \ggood/ for $\G$, a notion that we will introduce in Definition \ref{df:good} below.

Now we will use our pinning $(\B,\T,\{X_\alpha\}_{\alpha\in\Delta})$
to determine a point $x\in\mcB(\G)$ that is hyperspecial over $E$.
First, our pinning can be extended to a Chevalley system.
Bruhat--Tits \cite[\Sec4]{BT2}
associates to such a system
a valuation of the root groups for $\G$ over $E$.
Thus from \cite[\Sec6.2]{BT1},
we obtain a hyperspecial point $x$ in the apartment $\mcA(\T,E)$ of $\T$
in the building $\mcB(\G,E)$,
independent of the choice of the
Chevalley system extending our fixed
pinning.
Since our pinning is invariant under $\Gal(E/F)$,
we have that $x \in \mcA(\T,E)^{\Gal(E/F)} = \mcA(\T, F) \subseteq \mcB(\G)$
(here $\mcA(\T, F)$ means $\mcA(\S,F)$,
where $\S$ is the maximal $F$-split subtorus of $\T$).
It is easy to see that $x$ is independent of the choice of $E$. Henceforth,
we will be working with our fixed pinning, but the field $E$ will not be fixed.

\subsection{The set up}
\label{set up}
We continue to assume that $\G$ is quasi-split over $F$,
and from now on fix $\B$, $\T$, $\{X_{\alpha}\}_{\alpha \in \Delta}$,
$\B^-$, $\N^-$, $\lambda$, $Y$,
and hence obtain $\g(j)$ as in Notation \ref{new lambda weights},
except that these objects are now all defined over $F$.
Starting from
$(\B, \T, \{X_{\alpha}\})$, the maps
of \eqref{pinning orbit model}
give a regular nilpotent element of $\g(F)$, which coincides with $Y$, and
a point $x \in \mcB(\G)$
that becomes hyperspecial over any finite extension $E$ of $F$
that splits $\T$.

Thus, for any such $E$,
$x$ determines a model for $\G_E$ over $\mO_E$,
which is also the $\mO_E$-model defined by
a choice of a Chevalley basis over $E$ associated to the pinning
$(\B, \T, \{X_{\alpha}\})$.

We will abuse notation by letting
$\G$ also stand for this model.
This will not create any confusion,
as $\G(R)$ will still have a well-defined meaning when
$R$ is both an $F$-algebra and an $\mO_E$-algebra.
Similarly, we will
use $\T$ to also denote its obvious model over $\mO_E$ arising from
any identification of its base change to $E$ with a product of copies of
$\Gm$ over $E$.

\subsection{Conditions on the residual characteristic} \label{assumptions on p}
We are interested in conditions on $p$ that ensure that the adjoint
action induced by $Y$ on $\g_{\kappa_E}$ is smoothly regular.
For some purposes, we only require something weaker.

\begin{df}
\label{df:good}
\begin{enumerate}[(a)]
\item
Say that $p$ is \emph{\ngood/ (for $\G$)} if
the restriction of $\ad Y$
from the $\kappa_E$-vector space $\g(\kappa_E)$
to $\n^-(\kappa_E)$ has rank equal to $\dim_F \n^-$.
Equivalently, $[Y, \n^-_{\mO_E}]$
has a $\lambda$-invariant complement in $\b^-_{\mO_E}$
of rank $\dim_F \b^- - \dim_F \n^- \, (= \rk \G)$.
\item
Say that $p$ is \emph{\gFgood/ (for $\G$)} if for every real number $r$,
the map
$\overline{\ad Y}:
\fg_{x,r}(j)/\fg_{x,r+}(j) \rightarrow \fg_{x,r}(j+2)/\fg_{x,r+}(j+2)$
induced by $\ad Y$ is injective for every $j<0$ and surjective for $j\geq 0$.
\item
Say that $p$ is \emph{\ggood/ (for $\G$)} if
$p$ is \gEgood/ for $\G$.
Equivalently,
the endomorphism $\ad Y$ on the $\kappa_E$-vector space 
$\g(\kappa_E)$ has rank equal to $\dim \G - \rk \G$.
Equivalently,
$[Y, \g_{\mO_E}]$ has a $\lambda$-invariant complement
in $\g_{\mO_E}$ of rank equal to $\rk \G$.
\end{enumerate}
\end{df}

We will see in Corollary \ref{cor:gFgood} below how the conditions
of \ggood/ and \gFgood/ compare.
In particular, they are equivalent for tamely ramified groups $\G$.

\begin{remark} \label{sufficient condition for good prime}
Whether or not $p$ is \ggood/ (or \ngood/) for $\G$
depends only on the absolute root system for $\G$
unless $\G$ has a factor of type $A_n$ or $C_n$, in which case
it depends on the absolute root datum.
\begin{enumerate}[(a)]
\item
We have that $p$ is \ngood/ if and only if
$\G$ satisfies Hypothesis \ref{hyp for Kostant section}
with $R = \mO_E$.
Therefore,
Remark \ref{exclusions in hyp for Kostant section}
describes precisely which values
of $p$ are \ngood/.
\item
From
\cite[Theorem 52]{Witte17}
(and earlier \cite[Theorem 5.9]{Spr66} in the case where $\G$
is semisimple),
$p$ is \ggood/ if and only if
\begin{itemize}
\item
$p$ is \ngood/;
\item
$p\neq 2$ (resp.\ $3$)
if $\G$ has a factor of type
$C_n$ (resp.\ $G_2$);
and
\item
$p$ does not divide the 
order of $\pi_1(\G_{\der})$
or the order (in the scheme-theoretic sense)
of the component group of the center.
\end{itemize}
Note that this last condition is superfluous unless $\G$
has a factor of type $A_n$,
and all values of $p\geq 0$ are \ggood/
for a general linear group.
\end{enumerate}
\end{remark}

Rather than assuming that $\G$ splits over a tamely ramified
extension of $F$, we will sometimes be interested in situations
where the following weaker condition is met.

\begin{df}
\label{df:esstame}
Say that $\G$ is \emph{\esstame/}
if it satisfies the following:
\begin{enumerate}[(i)]
\item
Write $\G_{\scn} = \prod_i \Res_{E_i/F} \H_i$,
where each $\H_i$ is an absolutely almost simple, simply connected
group over a finite, separable extension $E_i$ of $F$.
Then each group $\H_i$ is split over a tamely ramified extension of $E_i$.
\item
Whenever $E'$ is any extension that splits $\G$,
then
$\t_{\der}(\mO_{E'})$
has a $\Gal(E'/F)$-invariant complement in $\t(\mO_{E'})$.
\end{enumerate}
\end{df}

Even though we will not use it,
we mention in passing that it is easy to check from the proof of
Lemma \ref{Kostant section over F} below that if
$p$ is \ngood/ and
$E'/F$ can be chosen to be tamely ramified,
then part (ii) of the definition holds, hence $\G$ is \esstame/.

If $p$ is \ngood/, then $\G$ satisfies part (i) of the definition
except possibly if $p=2$ and some $\H_i$ is of type $A_n$ ($n>1$),
or if $p=3$ and some $\H_i$ is a triality form of $D_4$.
Part (ii) of course needs to be verified only over
the minimal extension of $F$ in $\Fbar$ that splits $\T$,
and is automatic when $\G$ is semisimple.

\begin{remark}
Here is an example illustrating that Condition (ii)
in Definition \ref{df:esstame} does not follow from Condition (i)
of the same definition. Let $F = \QQ_2, E = \QQ_2[\sqrt{2}]$.
Note that there is an embedding $\mu_2 \hookrightarrow \Res_{E/F} \Gm$,
which at the level of $R$-points is given by $a \mapsto 1 \otimes a \in
(E \otimes_F R)^{\times} = \Res_{E/F} \Gm(R)$. $\mu_2$ also embeds
into $\SL_2$ as its center. Let $\G = (\Res_{E/F} \Gm \times \SL_2)/\Delta(\mu_2)$,
where $\Delta$ stands for the diagonal embedding. Since $\SL_2$ is split,
Condition (i) of Definition \ref{df:esstame} is automatically satisfied.
We claim that Condition (ii) is not.
The lattice $\X_*(\Res_{E/F} \Gm)$ has a basis
$\{e_2', e_3'\}$ permuted by $\Gal(E/F)$, and let $e_1'$ denote a basis
element for the cocharacter lattice of the standard maximal torus $\T_{\scn}$ of
$\SL_2$. Then we have an identification:
\[
\X_*(\T) = \{(a, b, c) \in ((1/2) \ZZ)^3 \mid a \equiv b \equiv c \mod \ZZ\},
\]
so that $\X_*(\T)$ is spanned by $e_1, e_2, e_3$ where
$e_1 := e_1', e_2 := (1/2)(e_1' + e_2' + e_3')$ and $e_3 := e_3'$. These
also give a basis for $\t(E)_0 = \X_*(\T) \otimes_{\ZZ} \mO_E$.
We claim that $\t_{\scn}(E)_0 = \mO_E e_1$ does not have a $\Gal(E/F)$-invariant complement in
$\t(E)_0$. To see this, we pass to $\kappa_E = \kappa_F$, upon which the nontrivial
element $\sigma$ of $\Gal(E/F)$ becomes a unipotent matrix $T_{\sigma}$ that
fixes the images $\bar e_1$ and $\bar e_2$ of $e_1$ and $e_2$, and takes the
image $\bar e_3$ of $e_3$ to $\bar e_1 + \bar e_3$. We conclude that
$T_{\sigma} - 1$ has $\bar e_1$ in its image, which
would have been impossible if $\kappa \bar e_1$ had a $T_{\sigma}$-invariant complement.
\end{remark}

\subsection{Kostant Sections over $F$ and $\mO_E$} \label{Kostant sections over F OE} 
Until the end of the proof of Lemma \ref{Kostant section over F} below,
fix a finite Galois extension $E$ over which $\G$ splits.
Let $K$ be the subextension of $E$ such that
$E/K$ is totally ramified and $K/F$ is unramified.
Recall that we will sometimes view $\G$, $\B$, $\T$, etc.,
also as groups defined over
$\mO_E$, using the Chevalley basis that we have fixed in
\Sec\ref{set up}.
For every $r \in \RR$, we have $\G(F)$-invariant subsets
$\g_r$ (resp., $\g_{r+}$) in $\g(F)$, defined as the union of
Moy--Prasad filtration sublattices $\g_{y, r}$ (resp., $\g_{y, r+}$)
as $y$ ranges over $\mcB(\G)$.
These filtration sublattices are normalized as in 
\cite[\Sec2.1.2]{AD02},
or equivalently as in \cite[\Sec1.4]{AD04}, unlike in \cite{RY14}, so
that $\g_{y, r + 1} = \varpi \g_{y, r}$ for all $y \in \mcB(\G)$.
Let $e$ be the ramification degree of $E/F$, so that $[E : K] = e$.
Choose
$0 = r_0 < r_1 < \dots < r_m = 1$
such that
\begin{equation} \label{ri's}
\g_{x, r_i} \supsetneq \g_{x, r_i+} = \g_{x, r_{i + 1}},
\quad\text{for all}\quad 
0 \leq i < m.
\end{equation}

\begin{remark} \label{Recall from Ad98}
\begin{enumerate}[(i)]
\item
If $\G$ is \esstame/, then
for all $r \in \RR$, $\g_{x, r} = \g(E)_{x, re}
\cap \g(F)$ and $\g(K)_{x, r} = \g(E)_{x, re} \cap \g(K)$
(see Proposition 1.4.1 of \cite{Ad98} and Lemma 3.14 of \cite{BKV16}).
We remark in passing that this assertion does not need the full strength of
$\G$ being \esstame/:
If we write $\G_{\scn} = \prod_i \Res_{E_i/F} \H_i$ as
in Definition \ref{df:esstame}(i), then it suffices to assume that
each $\H_i$ is not a wild special unitary group (i.e., either $\H_i$ is
not a special unitary group or it splits over a tamely ramified extension). 
Since $\Gal(E/F)$ fixes the point $x$, it also preserves
$\g(E)_{x, r}$ for all $r$.
\item
Since $x$ is hyperspecial over $E$,
if $\G$ is \esstame/, then (i) gives that
for each $i$, $r_i = j_i/e$ for some $j_i \in \{0, \dots, e\}$.
\item
For all $r \in \RR$, since
$\g_{x, r} = \g(K)_{x, r} \cap \g(F)$,
\'{e}tale descent gives that the obvious map
$\g_{x, r} \otimes_{\mO} \mO_K \longrightarrow \g(K)_{x, r}$
is an isomorphism. 
Therefore, the analogue of
\eqref{ri's} over $K$ holds
with the same numbers $r_i$, so that
\begin{equation} \label{ri's over K}
\g(K)_{x, r_i} \supsetneq \g(K)_{x, r_i+} = \g(K)_{x, r_{i + 1}},
	\, \forall \, 0 \leq i < m.
\end{equation}
\end{enumerate}
\end{remark}

Recall that $Y = \sum_{\alpha \in \Delta(\B, \T)} X_{\alpha} \in \g(F)$
is regular nilpotent, and belongs to $\g(E)_{x, 0} \cap \g(F) = \g_{x, 0}$.
Assuming $\G$ to be \esstame/,
we have a map $\g(K)_{x, r_i} \longrightarrow \g(E)_{x, 0}$ given
by $X \mapsto \varpi_E^{-j_i} X$.
This induces a $\kappa_K = \kappa_E$-linear map
$$
\xi_i \colon \g(K)_{x, r_i}/\g(K)_{x, r_i+}
\longrightarrow
\g(E)_{x, 0}/ \g(E)_{x, 0+}
$$
that is equivariant for the map induced on either
side by $\ad Y$
(which makes sense since
$[\g(K)_{x, r}, \g(K)_{x, s}] \subset \g(K)_{x, r + s}$
and $[\g(E)_{x, r}, \g(E)_{x, s}]
\subset \g(E)_{x, r + s}$ for all $r, s \in \RR$).

\begin{notn}
For a subset $L \subset \g(\Fbar)$, write $L(j)$ for $L \cap \g(j)(\Fbar)$
(see Notation \ref{new lambda weights}).
\end{notn}
From Equation \eqref{new lambda weight spaces} (recall that $x$ is
hyperspecial over $E$) we get that for all $r \in \RR$,
\[ \g(E)_{x, r} = \bigoplus_{j \in \ZZ} \g(E)_{x, r}(j). \]
Since $\g_{x, r}$ is a sum of $\mO$-submodules of eigenspaces for the adjoint action of the maximal $F$-split subtorus of $\T$, and similarly
with $\g(K)_{x, r}$, we have:
\begin{equation} \label{Moy Prasad lattices weight spaces}
\g(K)_{x, r} = \bigoplus_{j \in \ZZ} \g(K)_{x, r}(j), \text{\ \ and \ \ }
\g_{x, r} = \bigoplus_{j \in \ZZ} \g_{x, r}(j).
\end{equation}

\begin{remark} \label{rmk: on gFgood}
Equation \eqref{Moy Prasad lattices weight spaces}, together with
the fact that $\varpi \g_{x, r} = \g_{x, r + 1}$ for all
$r \in \RR$, allows
us to rephrase the notion of \gFgood/ as follows.
$p$ is \gFgood/ if and only if the rank
of the $\kappa$-vector space endomorphism:
\[
\overline{\ad Y}: \bigoplus_{\substack{0 \leq r < 1 \\ \g_{x, r} \neq \g_{x, r+}}}
\g_{x, r}/\g_{x, r+} \rightarrow
\bigoplus_{\substack{0 \leq r < 1 \\ \g_{x, r} \neq \g_{x, r+}}}
\g_{x, r}/\g_{x, r+}
\]
induced by $\ad Y$ equals

\[
\dsum_{j \neq 0} \dsum_{\substack{0 \leq r < 1 \\ \g_{x, r} \neq \g_{x, r+}}}
\dim_{\kappa} \g_{x, r}(j)/\g_{x, r+}(j) = \dsum_{j \neq 0}
\dim_{\kappa} \g_{x, 0}(j)/\g_{x, 1}(j) = \dsum_{j \neq 0}
\rk_{\mO} \g_{x, 0}(j) = \dsum_{j \neq 0} \dim_F \g(F)(j), \]
which equals $\dim \G - \rk \G$.
\end{remark}

\begin{lm} \label{Kostant section over F}
If $\G$ is \esstame/, each $\xi_i$ is injective.
Moreover, if $E/F$ is tamely ramified, the map
\begin{equation} \label{xi isomorphism}
\xi :=
\bigoplus_{i = 0}^{m - 1}
\xi_i :
\bigoplus_{i = 0}^{m - 1} \g(K)_{x, r_i}/\g(K)_{x, r_i+}
\longrightarrow \g(E)_{x, 0}/ \g(E)_{x, 0+}
\end{equation}
is an isomorphism of vector spaces over $\kappa_K = \kappa_E$ that
is equivariant for the action induced on both sides
by $\ad Y$.
\end{lm}

\begin{proof}
Since $\G$ is \esstame/, Remark \ref{Recall from Ad98}(i) applies.
The map
$\xi_i$ is injective because if $X \in \g(K)_{x, r_i}$
satisfies $\varpi_E^{-j_i} X \in \g(E)_{x, \varepsilon}$ for some
$\varepsilon > 0$, then
\[
X
\in \varpi_E^{j_i} \g(E)_{x, \varepsilon} \cap \g(K)
= \g(E)_{x, j_i + \varepsilon} \cap \g(K)
= \g(K)_{x, (j_i/e) + (\varepsilon/e)}
\subset \g(K)_{x, r_i +}.
\]
The domain and codomain of $\xi$
have the same dimension, namely $\dim \G$.
Now it suffices to show the linear disjointness of the images of the
$\xi_i$.
It is
easy to see that each $\sigma \in \Gal(E/K)$ induces
a $\kappa_K$-linear transformation on
$\g(E)_{x, 0}/ \g(E)_{x, 0+}$
that acts by $(\varpi_E \sigma(\varpi_E)^{-1})^{j_i}$
on the image of $\xi_i$. Since the characters $\sigma \mapsto
(\varpi_E \sigma(\varpi_E)^{-1})^{j}$ of $\Gal(E/K)$, $0 \leq j < e$,
are all distinct,
the linear disjointness of the images of the $\xi_i$
follows, and hence so does the lemma.
\end{proof}

\begin{remark}
If $E/F$ is wildly ramified, then $\xi$ need not be an isomorphism.
For example, consider the case when $\G = \Res_{\QQ_2[\sqrt{2}]/\QQ_2} \Gm$. 
\end{remark}

\begin{lm} \label{for extension tamely ramified}
Suppose $E$ is any Galois extension of $F$ splitting $\G$.
Give $\G$ an $\mO_E$-structure using our fixed pinning.
Then the following are equivalent:
\begin{enumerate}[(I)]
	\item There exists a finitely generated $\Gal(E/F)$-invariant
	$\mO_E$-submodule $L_{\mO_E} \subset \g(\mO_E)$
	such that $L_{\mO_E} = \oplus_{j \in \bZ} L_{\mO_E}(j)$
	is a $\lambda$-invariant complement in $\b^-(\mO_E)$
	for $[Y, \n^-(\mO_E)]$ of rank $\rk \G$.
	\item $\G$ is \esstame/ and $p$ is \ngood/.
\end{enumerate}
If the equivalent properties (I) and (II) are satisfied,
then for any lattice $L_{\mO_E}$ as in (I),
the following additional properties are equivalent:
	\begin{enumerate}[(I)]
		\addtocounter{enumi}{2}
		\item $\g(E)_{x, 0} = [Y, \g(E)_{x, 0}] + L_{\mO_E} + \fg(E)_{x,0+}$.
		\item $p$ is \ggood/.
	\end{enumerate}
\end{lm}

\begin{remark} \label{for extension tamely ramified base change}
The validity of Condition (I) of Lemma \ref{for extension tamely ramified}
does not depend on the choice of $E$.
To see this,
let $E \subset E'$ be an inclusion of
finite Galois extensions of $F$ in $\Fbar$ splitting $\G$.
Our pinning gives models $\G_{\mO_E}$ and
$\G_{\mO_{E'}}$ for $\G$ over $\mO_E$ and $\mO_{E'}$.
Given a lattice $L_{\mO_E}$ satisfying Condition (I),
note that $L_{\mO_E} \otimes_{\mO_E} \mO_{E'}$
satisfies the analogous condition over $E'$.
Conversely,
given a lattice $L_{\mO_{E'}}$ that satisfies the analogue of Condition (I)
over $E'$,
the lattice $L_{\mO_{E'}}^{\Gal(E'/E)}$ satisfies Condition (I).
Similarly, the validity of Condition (III)
does not depend on the choice of $E$ either.
\end{remark}

\begin{proof}[of Lemma \ref{for extension tamely ramified}]

\textit{Proof of (I) being equivalent to (II):}
Suppose for this paragraph that (I) holds. Then $p$ is \ngood/ (see the second sentence
of Definition \ref{df:good}(a)). Moreover, $p$ being \ngood/ forces
$[Y, \n^-(\kappa)] \cap \t(\kappa) = \t_{\der}(\kappa)$ by
dimension considerations, hence
$[Y, \n^-(\mO_E)] \cap \t(\mO_E) = \t_{\der}(\mO_E)$ by Nakayama's
lemma, so that $L_{\mO_E}$ furnishes a $\Gal(E/F)$-invariant complement
to $\t_{\der}(\mO_E)$
in $\t(\mO_E)$. Thanks to Remark \ref{for extension tamely ramified
base change}, this means that Condition \ref{df:esstame}(ii) holds.

Thus, we now suspend the assumption that
(I) holds, and assume these two conditions
($p$ being \ngood/ and Condition \ref{df:esstame}(ii) holding)
for the remainder of the proof, because otherwise neither (I) nor (II) is satisfied.
We present the remainder of the proof
in three steps. 

\vspace{\baselineskip}  

\textit{Step 1.  Proof in the tamely ramified case.}
Suppose $\G$ splits over a tamely ramified extension,
which may and shall be taken to be Galois over $F$.
By Remark \ref{for extension tamely ramified base change}, we
may take $E/F$ to be such an extension.

In this case, Condition \ref{df:esstame}(i) is automatic,  
and we need to construct
$L_{\mO_E}=\bigoplus_{j \leq 0}L_{\mO_E}(j)$ satisfying Property (I).

We let $L_{\mO_E}(0)$ be a $\Gal(E/F)$-equivariant complement
of $\t_{\der}(\mO_{E})$ in $\t(\mO_{E})$,
which exists by Condition \ref{df:esstame}(ii).

Next let $j < 0$.
Let $i \in \{0, \dots, m - 1\}$.
Then $\g_{x, r_i}$ and $\g_{x, r_i+}$ decompose as $\bigoplus_{\ell} \g_{x, r_i}(\ell)$
and $\bigoplus_{\ell} \g_{x, r_i+}(\ell)$, compatibly. Consider the map
\begin{equation} \label{j map}
\g_{x, r_i}(j - 2)/\g_{x, r_i+}(j - 2) \rightarrow \g_{x, r_i}(j)/\g_{x, r_i+}(j)
\end{equation} 
induced by $\ad Y$. Choose a $\kappa$-basis for a complement of the image of this map,
and lift it to a subset $A_{i, j} \subset \g_{x, r_i}(j)$.
Take $L_{\mO_E}(j)$ to be the $\mO_E$-span of $\varpi_E^{-j_i} A_{i, j}$
($i \in \{0, \dots, m - 1\}$). It is clearly a $\Gal(E/F)$-invariant,
finitely generated $\mO_E$-submodule
of $\g(E)_{x, 0}$.  From Remark \ref{Recall from Ad98}(iii),
Lemma \ref{Kostant section over F}, and
the fact that
$\dim_{\kappa_E} [Y, \g(j)(\kappa_E)] = \dim_E [Y, \g(E)(j)]$
(which follows since $p$ is \ngood/),
it follows that
$L_{\mO_E}(j)$ is indeed a complement for $[Y, \g(E)_{x, 0}(j - 2)]$
in $\g(E)_{x, 0}(j)$. 
Hence $L_{\mO_E}=\bigoplus_{j \leq 0}L_{\mO_E}(j)$ is a $\lambda$-invariant complement in
$\b^-(\mO_E)$ for $[Y, \n^-(\mO_E)]$.

\vspace{\baselineskip}  

\textit{Step 2:
Proof when Condition \ref{df:esstame}(i) is satisfied.}
We need to show the existence of
$L_{\mO_E}=\bigoplus_{j \leq 0}L_{\mO_E}(j)$ as before.
Again, we let $L_{\mO_E}(0)$ be a $\Gal(E/F)$-equivariant complement of $\t_{\der}(\mO_{E})$ in $\t(\mO_{E})$, which exists by
Condition \ref{df:esstame}(ii).

In order to define $L_{\mO_E}(j)$ for $j<0$, recall that we write $\G_{\scn}$
$= \prod_i \Res_{E_i'/F} \H_i$, where each $\H_i$ splits over a tamely ramified extension of $E_i'$ by Condition \ref{df:esstame}(i).
Thus, as a $\Gal(E/F)$-$\mO_E$-module, we have for $j \neq 0$:
\begin{equation*} 
	\fg(E)(j) \simeq
	\bigoplus_i \bigoplus_{\sigma\in\mcY_i} \sigma \cdot \fh_i(E)(j),
\end{equation*}
and
$Y=\sum\limits_{i}\sum\limits_{\sigma \in \mcY_i}\sigma \cdot Y_i$,
where
$Y_i=\sum\limits_{\alpha \in \Delta_i}X_{\alpha}$
with $\Delta_i$ being a basis of the roots of $\H_i$ compatible with $\Delta$,
and $\mcY_i$ denotes a set of representatives of
$\Gal(E/F)/\Gal(E/E_i')$.

By Step 1, for each $i$, there exists a lattice $L_{\H_i,\mO_{E'_i}}(j) \subset \fh_i(E)(j)$
that is a $\Gal(E/E'_i)$-invariant complement of $[Y_i,\fh_i(E)(j-2)]$. Define
\begin{equation*} 
	L_{\mO_E}(j):=
	\bigoplus_i \bigoplus_{\sigma\in\mcY_i}
		\sigma \cdot L_{\H_i,\mO_{E'_i}}(j).
\end{equation*}
Let $L_{\mO_E}:=\bigoplus_{j \leq 0}L_{\mO_E}(j)$.
Then $L_{\mO_E}$ is $\Gal(E/F)$-invariant and satisfies (I) by construction.

\vspace{\baselineskip}  

\textit{Step 3:
Proof when Condition \ref{df:esstame}(i) is not satisfied.}
There are two situations to treat here:
\begin{itemize}
	\item one of the $\H_i$'s is of
	type $A_n$, splitting over only a ramified quadratic extension
(in particular, $n \geq 2$), and $p = 2$; and
	\item one of the $\H_i$'s is of type $D_4$,
	on which the action of the absolute Galois group of
	$E_i'$ has order at least three, and $p = 3$.
\end{itemize}

We consider the case $\G=\H_1$ first and deduce the more general case afterwards.

Let us consider the $A_n$ case first, split over a ramified quadratic
extension $E$ over $F$.
In this case,
$n \geq 2$,
$\g({\mO_E})(-2n)$ is of rank 1,
and $\g({\mO_E})(-2n + 2)$ is of rank 2.
Further, $\g({\mO_E})(-2n + 2)$
is the direct sum of $[Y, \g({\mO_E})(-2n)]$ and another rank-one
$\mO_E$-submodule (all the assertions so far are a $\GL_n$-computation or a computation using a Chevalley system, depending on the reader's preference). Thus,
analogous results hold over $\kappa_E$, too.
The group
$\Gal(E/F)$ exchanges the root spaces spanning
$\g({\mO_E})(-2n + 2)$, as it exchanges the corresponding roots and
induces an automorphism of $\g(E)_{x, 0}$. Thus,
we find that the $\kappa_E$-vector
space $\g({\kappa_E})(-2n + 2)$ is spanned by two vectors $e_1, e_2$, which
are exchanged by the $\kappa_E$-linear automorphism $T_{\sigma}$ induced by
the unique nontrivial element $\sigma \in \Gal(E/F)$.
Since $p = 2$,
$T_{\sigma}$ is unipotent.
If $[Y, \g({\mO_E})(-2n)]$ had a $\Gal(E/F)$-invariant complement in
$\g({\mO_E})(-2n + 2)$, then $[\bar Y, \g({\kappa_E})(-2n)]$ would have
a $T_{\sigma}$-invariant complement in $\g({\kappa_E})(-2n + 2)$.
But $[\bar Y, \g({\kappa_E})(-2n)]$ is a one-dimensional $\kappa_E$-vector
space
(since $p$ is \ngood/, it is enough
to check the analogous result for a general linear group over $\CC$).
Thus, the existence of a $T_{\sigma}$-invariant complement would 
imply that
$T_{\sigma}$ was a nontrivial semisimple automorphism of
$\g({\kappa_E})(-2n + 2)$, a contradiction.

Now we consider the $D_4$ case. Here the argument is somewhat similar.
We consider $[Y, \g({\mO_E})(-8)] \subset \g({\mO_E})(-6)$.
If $\alpha_1$, $\alpha_2$, $\alpha_3$,
and $\alpha_4$ are the simple roots of $D_4$, with $\alpha_2$ the unique
root fixed by $\Gal(\Fbar/F)$, then
$\g({\mO_E})(-8)$ is spanned by the root space
corresponding to
$-(\alpha_1 + \alpha_2 + \alpha_3 + \alpha_4)$, while $\g({\mO_E})(-6)$ is spanned by the root
spaces corresponding to the $-(\alpha_1 + \alpha_2 + \alpha_3 + \alpha_4) +
\alpha_i$ as $i$ ranges over $1, 3$ and $4$. $[Y, \g({\mO_E})(-8)]$ is
a rank one $\mO_E$-submodule of $\g({\mO_E})(-6)$, admitting
a complementary rank two $\mO_E$-submodule (use that $p$ is \ngood/).
As before, we
have an automorphism $T_{\sigma}$ of $\g({\kappa_E})(-6)$, which
permutes three basis vectors in a three cycle and hence can be checked
to have minimal polynomial $(T - 1)^3$ (thanks to $p = 3$). But
if the image of $[Y, \g({\mO_E})(-8)]$ in $\g({\kappa_E})$ had a 
$T_{\sigma}$-invariant complement in $\g({\kappa_E})(-6)$, then
the minimal polynomial of $T_{\sigma}$ on $\g({\kappa_E})(-6)$ would
have to be $(T - 1)^2$, a contradiction.

Now drop our assumption that $\G = \H_1$, and write
$\G_{\scn}=\prod_i \Res_{E_i'/F} \H_i$, where
$\H_1$ is of one of the two types mentioned at the beginning of
Step 3, and $p=2$ or $p=3$ as appropriate.
Let $j=-2n+2$ if $p=2$, and $j=-6$ if $p=3$.
To get a contradiction, 
assume that for some splitting field $E$ of $F$, $[Y, \n^-(\mO_E)]$ has a
$\Gal(E/F)$-invariant $\lambda$-invariant complement in $\b^-(\mO_E)$,
i.e., in particular,  $[Y, \fg(\mO_E)(j-2)]$ has a
$\Gal(E/F)$-invariant complement in $\fg(\mO_E)(j)$.
Recall that
$\fg(\mO_E)(j)
\simeq
\bigoplus_i \bigoplus_{\mcY_i} \sigma \cdot \fh_i(\mO_E)(j)$,
and $\Gal(E/F)$ preserves each summand
$\bigoplus_{\mcY_i} \sigma \cdot \fh_i(\mO_E)(j)$.
Hence, using the notation from Step 2,
we obtain a $\Gal(E/F)$-invariant complement of the image of $\ad(\sum\limits_{\sigma \in \mcY_i}\sigma \cdot Y_i)$
in $\bigoplus_{\mcY_i} \sigma \cdot \fh_i(\mO_E)(j)$.
By restriction to $\fh_1(\mO_E)(j)$
we obtain a $\Gal(E/E_1')$-invariant complement to $[Y_1,\fh_1(\mO_E)(j-2)]$
in $\fh_1(\mO_E)(j)$.
This yields a contradiction by the above-treated absolutely almost simple cases.

\vspace{\baselineskip}  

\textit{Proof of (III) being equivalent to (IV), if (I) and (II) hold:}
First note that given a module $L_{\mO_E}$
satisfying Property (III) of Lemma \ref{for extension tamely ramified},
we may look at the image of either side of the equality of (III)
inside $\g(E)_{x, 0}/\g(E)_{x, 0+} = \g(\kappa_E)$ to get
\[ \dim_{\kappa_E} \g(\kappa_E) - \dim_{\kappa_E} [Y, \g(\kappa_E)] =
\rk_{\mO_E} L_{\mO_E} = \rk \G, \]
so that $p$ is \ggood/.
Thus we now assume that $p$ is \ggood/.
As the image $\overline L_{\mO_E}$ of $L_{\mO_E}$ in
$\fg(E)_{x,0}/\fg(E)_{x,0+}\simeq \fg(\kappa_E)$
is a complement in $\b^-(\kappa_E)$
for $[\overline Y,\n^-(\kappa_E)]$, we deduce
(by Definition \ref{df:good} and dimension counting)
that $\overline L_{\mO_E}$
is a complement in $\fg(\kappa_E)$
for $[\overline Y,\fg(\kappa_E)]$,
and hence Property (III) follows.
\end{proof}

\begin{cor}\label{cor:gFgood}
	Let $\G$ be \esstame/. If $p$ is \ggood/ and $\ov{\ad Y}:\fg_{x,r}(0)/\fg_{x,r+}(0) \rightarrow \fg_{x,r}(2)/\fg_{x,r+}(2)$ is surjective for all $r \in \mathbb R$, then $p$ is \gFgood/.

	Moreover, if $\G$ is tamely ramified, then $p$ is \gFgood/ if and only if $p$ is \ggood/.
	
\end{cor}

\begin{proof}
If $E/F$ is tame,
then the claim follows from Lemma \ref{Kostant section over F}
(whose isomorphism preserves the $\Ad \circ \lambda$-eigenspaces)
and the fact that the map
$\g_{x, r}/\g_{x, r+} \rightarrow (\g(K)_{x, r}/\g(K)_{x, r+})^{\Gal(K/F)}$
is an isomorphism.
	
Suppose now that $p$ is \ggood/.
Then we have for $j<0$ that $\ov{\ad Y}$ sends
$\fg(E)_{x,0}(j)/\g(E)_{x,0+}(j)$ injectively to
$\fg(E)_{x,0}(j+2)/\g(E)_{x,0+}(j+2)$.
Using the injection $\xi_i$ from Lemma \ref{Kostant section over F}, we can embed $\g_{x,r}/\g_{x,r+}$
$\ov{\ad Y}$-equivariantly into $\g(E)_{x,0}/\g(E)_{x,0+}$, and this embedding preserves the decomposition into $\Ad \circ \lambda$-eigenspaces.
Hence the action $\ov{\ad Y}$
is injective on $\g_{x,r}(j)/\g_{x,r+}(j)$ for $j<0$.
Next we need to show that $\ov{\ad Y}$ maps $\fg_{x,r}(j)/\g_{x,r+}(j)$ surjectively onto $\fg_{x,r}(j+2)/\g_{x,r+}(j+2)$ for $j>0$.
To see this,
we use the notation from Step 2 of the proof of
Lemma \ref{for extension tamely ramified},
and let $\sum_{\sigma \in \mcY_i} \sigma \cdot X$
be an element of $\fg_{x,r}(j+2)$ ($i$ fixed),
where $X \in \fh_i(E_i')_{x_i,re_i'}(j+2) \subset \fg(E)_{x,re}(j+2)$,
where $e_i'$ is the ramification degree of $E_i'/F$,
and $x_i$ is the projection of $x$ into the building of $\H_i(E_i')$.
By the equivalence of \ggood/ and \gFgood/
for the tamely ramified group $\H_i$,
there exists $X'$ in $\fh_i(E_i')_{x_i,re_i'}(j)$
such that
$X+\fh_i(E_i')_{x_i,re_i'+}(j+2)=[Y_i,X']+\fh_i(E_i')_{x_i,re_i'+}(j+2)$,
and hence
$\sum_{\mcY_i} \sigma \cdot X+\fg_{x,r+}(j+2)=[Y,\sum_{\mcY_i} \sigma \cdot X']+\fg_{x,r+}(j+2)$.
As $\fg_{x,r}/\fg_{x,r+}$ is spanned by elements of the form $\sum_{\mcY_i} \sigma \cdot X + \fg_{x,r+}$, we deduce the claim.
Thus, if $\ov{\ad Y}:\fg_{x,r}(0)/\fg_{x,r+}(0) \rightarrow \fg_{x,r}(2)/\fg_{x,r+}(2)$ is surjective, we obtain that $p$ is \gFgood/, as desired.
\end{proof}

\begin{remark} \label{adjoint gFgood}
In Corollary \ref{cor:gFgood},
the condition that
$\overline{\ad Y}:
	\g_{x, r}(0)/\g_{x, r+}(0) \rightarrow \g_{x, r}(2)/\g_{x, r+}(2)$
be surjective is automatic if $\G$ is semisimple,
as we will explain below.
In fact, we do not need to assume that $\G$ is semisimple.
Instead, consider the identity component $\Z^0$ of the center of $\G$,
which is canonically an $\mO_E$-torus and hence gives its Lie algebra
$\z$ an $\mO_E$-scheme structure $\z_{\mO_E}$.
Let us assume that $\z_{\mO_E}(\mO_E)$ has an $\mO_E[\Gal(E/F)]$-invariant complement
$\s$ in $\t(\mO_E)$.
Of course, this is trivially satisfied if $\G$ is semisimple.

If we can show that the morphism
$\ad Y : \s \rightarrow \g(2)(\mO_E)$ of $\mO_E[\Gal(E/F)]$-modules
is an isomorphism, then, 
multiplying by
$\varpi_E^j$ for any given $j$ and taking $\Gal(E/F)$-invariants,
we would get from Remark \ref{Recall from Ad98}(i) that
$\ad Y(\g_{x, j/e}(0)) \supset \ad Y((\varpi_E^j \s)^{\Gal(E/F)}) = \g_{x, j/e}(2)$.
Since $j$ is arbitrary,
and since $\g_{x, r} = \g_{x, \lceil re \rceil /e}$ for all
$r$
(use Remark \ref{Recall from Ad98}(i) and that
$\g(E)_{x, re} = \g(E)_{x, \lceil re \rceil}$
because $x$ is hyperspecial over $E$),
it would then follow
that for all $r$,
$\overline{\ad Y}:
\g_{x, r}(0)/\g_{x, r+}(0) \rightarrow \g_{x, r}(2)/\g_{x, r+}(2)$
is surjective.
Thus, it suffices to show that the morphism
$\ad Y : \s \rightarrow \g(2)(\mO_E)$ of $\mO_E$-modules
is an isomorphism, or equivalently, surjective (since
$\rk_{\mO_E} \s = \rk_{\mO_E} \g(2)(\mO_E)$).
But this follows since the map $\ad Y : \g(0)(\mO_E) \rightarrow
\g(2)(\mO_E)$ is surjective ($p$ being \ggood/), with kernel
containing $\z_{\mO_E}(\mO_E)$.

Thus, if $\G$ is \esstame/ and semisimple (or, more generally,
$\z$ has a suitable complement in $\t$),
and $p$ is \ggood/, then $p$ is \gFgood/.
\end{remark}

\begin{remark} 
We remark that $p$ can be
\gFgood/ without being \ggood/. This occurs, for example, if $\G =
\Res_{E/F} (\SL_2 \times \Gm)/ \mu_2$ with $E =
\QQ_2[\sqrt{2}]$, $F = \QQ_2$, and where the embedding
$\mu_2 \hookrightarrow \Res_{E/F} (\SL_2 \times \Gm)$  is
obtained by composing the diagonal embedding
$\mu_2 \hookrightarrow \SL_2 \times \Gm$ with the `diagonal' inclusion
$\SL_2 \times \Gm \hookrightarrow \Res_{E/F}(\SL_2 \times \Gm)$
(which at the level of $R$-points, for an $F$-algebra $R$, corresponds to the inclusion
$(\SL_2 \times \Gm)(R) \hookrightarrow (\SL_2 \times \Gm)(E \otimes_F R)$
induced by the $R$-algebra embedding $R \hookrightarrow E \otimes_F R, r \mapsto 1 \otimes r$).
\end{remark}

\begin{cor} \label{extending from tamely ramified}
	Suppose that $\G$ is \esstame/ and $p$ is \ngood/.
	There exists a subspace $L_F \subset \b^-(F)$ such that:
	\begin{enumerate}[(i)]
		\item 
		$L_F$ is a $\lambda$-invariant complement in $\b^-(F)$ for $[Y, \n^-(F)]$. Hence $Y + L_F$ is a section for $\g \rightarrow \g \doubleslash \G$ over $F$.
		\item The set of topologically nilpotent elements in $Y + L_F$ equals
		$Y + L_{0+}$, with $L_{0+} := L_F \cap \g_{x, 0+}$.
	\end{enumerate}
	If we assume that $p$ is \gFgood/, then we can choose $L_F$ such that
	\begin{enumerate}[(i)]
		\addtocounter{enumi}{2}
		\item
		$\g_{x, r} = [Y, \g_{x, r}] + L_r + \fg_{x,r+}$
		for all $r \in \RR$, where $L_r := L_F \cap \g_{x, r}$.
	\end{enumerate}
\end{cor}

\begin{proof}
	Assuming $\G$ is \esstame/ and $p$ is \ngood/
	puts at our disposal a lattice $L_{\mO_E}$ satisfying (I)
	of Lemma \ref{for extension tamely ramified}.
	The $E$-span of $L_{\mO_E}$ gets an
	$F$-structure $L_F$ by Galois descent
	(see \cite[Proposition 11.1.4]{Spr98}),
	and Condition (i) is easy to check. Since $L_{\mO_E}$ satisfies condition
	(I) of Lemma \ref{for extension
		tamely ramified}, Proposition \ref{Kostant section pro} gives
	that $Y + L_{\mO_E} \rightarrow \g \doubleslash \G$ is an isomorphism
	of schemes over $\mO_E$ (where we use the $\mO_E$-structure arising from our
	fixed pinning). Therefore, Condition (ii) then follows from Lemma
	\ref{unramified tn}
	and Remark \ref{Recall from Ad98}(i), which applies as $\G$ is \esstame/.
	
	Suppose now that $p$ is \gFgood/.  We work with the $L_F$ constructed
above.
	We have $\g_{x, r} \supset [Y, \g_{x, r}] + L_r + \fg_{x,r+}$,
and hence $\dim_{\kappa} \g_{x,r}/\g_{x,r+} \geq  \dim_\kappa [Y, \g_{x, r}]/\g_{x,r+} + \dim_\kappa (L_r+\fg_{x,r+})/\fg_{x,r+}$ for all $r$,
because $[Y, \g_{x, r}]/\g_{x,r+}$ and $(L_r+\fg_{x,r+})/\fg_{x,r+}$ have trivial intersection 
by our construction of $L_F$. Since $p$ is
\gFgood/ we have (see Remark \ref{rmk: on gFgood}):
	\begin{eqnarray*}
		\sum_{i=1}^m \left(\dim_\kappa \g_{x, r_i}/\g_{x,r_i+} - \dim_\kappa [Y, \g_{x, r_i}/\g_{x,r_i+}] \right)&=&\sum_{i=1}^m \dim_\kappa \g_{x, r_i}(0)/\g_{x,r_i+}(0) \\
&=& \dim \t = \dim_F L_F \\
		& = & \sum_{i=1}^m \dim_\kappa (L_{r_i}+\fg_{x,r_i+})/\fg_{x,r_i+}.
	\end{eqnarray*}
Hence $\dim_{\kappa} \g_{x,r}/\g_{x,r+} =  \dim_\kappa [Y, \g_{x, r}/\g_{x,r+}] + \dim_\kappa (L_r+\fg_{x,r+})/\fg_{x,r+}$ and therefore $\g_{x, r} = [Y, \g_{x, r}] + L_r + \fg_{x,r+}$ for all $r \in \mathbb R$.
\end{proof}

\begin{notn}
Given $L_F$ as in
Corollary \ref{extending from tamely ramified},
and any $r \in \RR$, we will write $L_r$
for $L_F \cap \g_{x, r}$ and $L_{r+}$ for $L_F \cap \g_{x, r+}$,
continuing with (and slightly extending)
the notation introduced in (ii) and (iii)
of the corollary.
\end{notn}

\subsection{Relating $Y + \g_{x, r}$ with a Kostant Section} \label{relating with Kostant section}

If $K = E$, so that $\G$ is unramified, then $x$ is a hyperspecial
point of $\G$, realizing $\G$ as a reductive group over $\mO$.
In such a situation, we supply $\q$ with the $\mO$-structure
determined by the containment $\mO[\g]^{\G} \subset
F[\g]^{\G}$, and write $\q(m)$ for
the preimage of $\chi(0) \in \q(\mO/\varpi^m \mO)$
under the obvious map
$\q(\mO) \longrightarrow \q(\mO/\varpi^m \mO)$
(see Notation \ref{adjoint quotient notation}).

\begin{df}
\label{df:blah}
Say that $\G$ satisfies \emph{\blahblah/}
if there is a tame extension $E'/F$
such that
$\G_{E'}$ contains a maximal torus that is induced.
\end{df}

We've adapted terminology from \cite{Yu02}.

\begin{lm} \label{AD04 for unramified}
Assume that $\G$ is \esstame/
and satisfies \blahblah/.
\begin{enumerate}[(i)]
\item
Suppose $p$ is \gFgood/, so that we may and do choose $L_F$ as in
Corollary \ref{extending from tamely ramified}.
For $r > 0$, $Y + \g_{x, r} = \,^{\G_{x, r}}(Y + L_r)$ (where, for $J \subset \G(F)$
and $\Omega \subset \g(F)$, we write $\,^J\Omega$ for $\Ad(J)(\Omega)$);
\item In the situation of (i), suppose
$K = E$, so that $\G$ is unramified. Then for $m \in \NN$ we have
\[ \chi(Y + \g(F)_{x, m}) = \q(m). \]
\end{enumerate}
\end{lm}
\begin{proof}
Part (ii) follows from (i) together with
Proposition \ref{Kostant section pro}.

Let us prove (i) exactly as in
\cite[Lemma 5.2.1]{Deb02a}
(whose setting is far less restrictive, but
requires $p$ to be zero or large),
which is in turn inspired by
\cite[\Sec IX.4]{Wal01}.
Only for the reader's convenience,
we give the details.

Only the containment `$\subseteq$' is nontrivial.
By condition (iii) in Corollary \ref{extending from tamely ramified},
we have that for each $l$,
$\g_{x, l} = [Y, \g_{x, l}] + L_l + \fg_{x,l+}$.
Hence it is enough to prove that for each $l \geq r$,
\begin{equation} \label{for unramified AD04}
Y + L_r + [Y, \g_{x, l}] \subset \,^{\G_{x, l}}(Y +
L_r + \g_{x, l+}).
\end{equation}	
Let $Y + C + [Y, P]$ belong to the left-hand side, with
$C \in L_r$ and $P \in \g_{x, l}$. We wish to
find $h \in \G_{x, l}$ such that $\Ad h(Y + C + [Y, P])
\in Y + L_r + \g_{x, l+}$.

For this we use a mock exponential map.
Let $\varphi_l : \g_{x, l} \longrightarrow
\G_{x, l}$ be as constructed in
\cite[\Sec\Sec1.3--1.5]{Ad98}.
(Note that the assumption of $\G$ satisfying \blahblah/ is necessary
to ensure that such a map exists,
contrary to the claims in \cite{Ad98,MP96}.)
We set $h = \varphi_l(-P)$.
Since
\begin{multline*}
\Ad h(Y + C + [Y, P]) \\
=
Y + C+ (\Ad h(Y)-Y+[Y,P]) + (\Ad h(C) - C) + (\Ad h([Y,P])-[Y,P]),
\end{multline*}
it suffices to show that each of the 
three parenthetical terms on the right-hand side of
the above equation belongs to $\g_{x, l+}$.

This follows from \cite[Prop.\ 1.6.3]{Ad98},
together with the fact that $[\g_{x,a},\g_{x,b}]\subseteq \g_{x,a+b}$
for all $a,b\in \RR$.

\end{proof}

\subsection{The main result of \protect\cite{AD04} under our assumptions} \label{with AD04 main result}

The following result amounts to (a slight sharpening of) \cite[Proposition 1]{AD04} but with
much milder hypotheses.

\begin{pro} \label{AD04 main result} 
Assume $\G$ is \esstame/, and $p$ is \ngood/.
\begin{enumerate}[(i)]
\item 
Let $Z \in \g(F)$.
Then $Z$ is $\G(\Fbar)$-conjugate to an element of
$Y + \g_{x, 0+}$ if and only if it is regular and topologically nilpotent.
\item
If $\G$ satisfies \blahblah/, and $p$ is \gFgood/,
then for any regular $Z \in \g(F)_{\tn}$,
$\Ad \G(\Fbar)(Z) \cap (Y + \g_{x, 0+})$ is a single orbit under
$\G_{x, 0+}$.
\end{enumerate}
\end{pro}
\begin{proof}
Let $X\in Y+\fg_{x,0+} \subset \fg(\mO_E)$,
and let $C_{\G_{\mO_E}}(X)$ be
the centralizer scheme of $X$ in $\G_{\mO_E}$.
By \cite[Expos\'e VI${}_{\textrm{B}}$, Proposition~4.1]{SGA3-I},
$\dim C_{\G_{E}}(X)
= \dim C_{\G_{\mO_E}}(X)_E
\leq \dim C_{\G_{\mO_E}}(X)_{\kappa_E}
= \dim C_{\G_{\kappa_E}}(\ov Y)$,
where $\ov Y$ is the image of $Y$ in $\fg(\kappa_E)$.
Hence, since $\ov Y$ is regular (\cite[Lemma 3.1.1]{Ric16}),
we deduce that $X$ is regular.
Thus, if $Z$ is conjugate to an element in $Y+\g_{x,0+}$, then it is regular. 
Moreover, that every element of $Y + \g_{x, 0+}$ is topologically nilpotent
follows from base changing to a field over which $\G$ splits and applying
Lemma \ref{unramified tn} and Remark \ref{Recall from Ad98}(i).

Suppose now that $Z$ is regular.
Following the reasoning of \cite{Ric16} (around Equation (3.1.1)),
we see that $Y + L_F$ consists
entirely of regular elements.
More precisely, it follows from the facts that $Y$ is regular, that the locus of regular elements is open and that there is
a contracting action of $\mathbb{G}_m$ on $Y+(L_F \otimes_F E)$, which was described in the proof of Lemma \ref{a is injective}.
Thus, $Z$ is $\G(\Fbar)$-conjugate to the unique (regular) element
$Z' \in Y + L_F$ such that $\chi(Z) = \chi(Z')$.
Therefore, $Z \in \g(F)_{\tn}$ if and only if $Z' \in \g(F)_{\tn}$,
which by Condition (ii) in Corollary \ref{extending from tamely ramified} 
is equivalent to $Z' \in Y + L_{0+} \subset Y + \g_{x, 0+}$. This gives (i).

Lemma \ref{AD04 for unramified} now gives (ii): since the proof
of (i) shows that $Z \in \g(F)_{\tn}$ if and only if $\Ad \G(\Fbar)(Z)$ meets
$Y + L_{0+}$, the assertion that $\Ad \G(\Fbar)(Z) \cap (Y + \g_{x, 0+})$ is
a single $\G_{x, 0+}$-orbit is equivalent to saying that
$Y + \g_{x, 0+} = \,^{\G_{x, 0+}}(Y + L_{0+})$.
\end{proof}

\begin{remark} \label{main result for products}
Although we do not need it, we mention in passing that
Proposition \ref{AD04 main result} 
holds for slightly more general groups.
First, it
holds when $\G$ is an arbitrary torus
(i.e., not necessarily satisfying \blahblah/),
because of our definition
of ``topologically nilpotent.''
Second, if the proposition holds for two groups, then it holds for their direct
product.
Third,
if the proposition holds for a group, then it holds for the image of the
group under any isogeny whose schematic kernel has order not divisible
by $p$.
\end{remark}

\begin{remark} \label{reconciling topological nilpotence}
Assume that $p$ is \ggood/ and $\G$ is \esstame/.
Let $\g(F)_{\tn}'$ denote the 
union of the lattices $\g_{z, 0+}$ as $z$ varies over $\mcB(\G)$.
Recall that in \cite{AD04}, unlike in Definition \ref{top nilp},
an element of $\g(F)$ is called ``topologically nilpotent'' precisely
when it belongs to $\g(F)_{\tn}'$.
We now show that under our hypotheses, $\g(F)_{\tn} = \g(F)_{\tn}'$.
The set $\g(F)_{\tn}$ is open and closed in $\g(F)$,
as can be seen for instance by base changing to $E$ and using
Lemma \ref{unramified tn}.
To see that $\g(F)_{\tn}'$ is also open and closed in $\g(F)$,
note that in the definition of this set, we may restrict our union
to barycenters $z$ of alcoves in $\mcB(\G)$,
which shows that for some positive $\varepsilon>0$,
$\g(F)_{\tn}' = \g_\varepsilon$
in the notation of \cite[Corollary 3.4.3]{AD02},
which we can then apply.
Thus,
it suffices to show that
the set of regular semisimple elements in $\g(F)_{\tn}$
is the same as
the set of regular semisimple elements in $\g(F)_{\tn}'$
(here we are using that the regular semisimple elements are dense in $\g(F)$;
this is an easy consequence of the fact that
$d\alpha$ does not vanish for any root $\alpha$,
which follows, as mentioned in
\cite[Remark 2.2.1(1)]{Ric16}, from
$p$ being \ggood/).
If $Z \in \g(F)_{\tn}$ is regular semisimple, then
by Proposition \ref{AD04 main result}, $Z$ has a $\G(\Fbar)$-conjugate
in $Y + \g_{x, 0+}$, which by \cite[Corollary 3.2.6]{AD02}
is contained in $\g(F)_{\tn}'$.
Hence $Z \in \g(F)_{\tn}'$ as well,
since $\g(F)_{\tn}'$ is a union
of stable orbits.
(This follows from \cite[Lemma 8.5]{BKV16}.
Alternatively, use
an argument combining \cite[Hypothesis 3]{AD04} and
\cite[Lemma 2.2.5]{AD04crelle}, as in the proof of
the ``$\Rightarrow$'' implication of \cite[Proposition 1]{AD04}).
Conversely, suppose $Z \in \g(F)_{\tn}'$
is regular semisimple.
Then the identity component $\T_Z$ of the centralizer of $Z$ in $\G$
is a maximal torus.
For an extension $E_1$ of $F$ splitting
$\T_Z$, we have $Z \in \g(E_1)_{\tn}'$, so that
$Z \in \t_1(E_1)_{\tn}'$ by \cite[Theorem 3.1.2(2)]{AD02},
forcing $Z \in \g(E_1)_{\tn} \cap \g(F) = \g(F)_{\tn}$.
\end{remark}

\section{Some pairs of matching functions} \label{pairs of matching functions}
We now assume the following:
\begin{hyp} \label{hyp:Sec3}
\begin{enumerate}[(a)]
\item
\label{item:char0}
$F$ is a finite extension of $\QQ_p$.
\item
\label{item:g-good}
$p$ is \ggood/ for $\G$.
\item
\label{item:p-not-2}
$p \neq 2$.
\item
\label{item:esstame}
$\G$ satisfies \blahblah/ and is \esstame/.
Moreover, if $p = 3$, then, writing 
$\G_{\scn}$ as a product of groups
$\text{Res}_{E_i/F} \H_i$, with $\H_i$ an absolutely
almost simple group over $E_i$,
each group $\H_i$ splits over a quadratic extension of $E_i$.
\item
\label{item:isogeny}
$p$ does not divide the cardinality of the
center of $\G_{\scn}$.
\end{enumerate}
\end{hyp}
 Henceforth, the absolute value $|\cdot|$ on $F$ will be normalized,
and $|\cdot|$ will continue to denote its unique extension to $\Fbar$. For any finite extension $E'$ of $F$ in $\Fbar$, let $|\cdot|_{E'}$
denote the extension to $\Fbar$ of the normalized absolute value
on $E'$.

\begin{remark} \label{p is gFgood}
It follows from
Hypothesis \ref{hyp:Sec3}(\ref{item:isogeny})
that for any Galois extension $E/F$ over which $\G$ splits,
$\Lie(\Z^0)(\mO_E)$ has a $\Gal(E/F)$-invariant
complement over $\mO_E$.
Therefore,
by Remark \ref{adjoint gFgood},
$p$ is \gFgood/.
\end{remark}

\begin{remark}
\begin{enumerate}[(a)]
\item
We assume part (\ref{item:char0})
of Hypothesis \ref{hyp:Sec3}
not only to make it easier to handle
orbits and measures on orbits, but also because we have
not yet been able to locate references stating
a suitable level of generality in which \cite[Theorem 5.1]{Kot99},
which we will need later, may be applied.
\item
Parts (\ref{item:p-not-2}) and (\ref{item:isogeny})
are superfluous unless $\G$ has a factor of type $A_n$.
\item
Part (\ref{item:esstame})
only excludes a few cases beyond those already
excluded by other parts of Hypothesis \ref{hyp:Sec3}:
We require $p\neq 3$ if some $\H_i$ is an unramified triality form of $D_4$.
(This makes $p$, if positive,
a `very good prime' in the sense of
\cite[\Sec 8.10]{BKV16}.)
We require this hypothesis, along with 
parts (\ref{item:p-not-2}) and (\ref{item:isogeny}),
in order to apply a Kazhdan--Varshavsky quasi-logarithm.
\end{enumerate}
\end{remark}

Let $\H$ be a (necessarily quasi-split) group underlying a fixed endoscopic
datum for $\G$,
and assume that $\H$ satisfies Hypothesis \ref{hyp:Sec3} as well.
Henceforth, we fix a Galois extension $E/F$ that splits both
$\G$ and $\H$.

\subsection{Comments on measures} \label{comments on measures}
We will have to work with Haar measures specified using differential forms,
and follow the second paragraph of \cite[\Sec1.4]{LS87}.
For an algebraic
group $\G_1$ over $F$, recall that $\G_{1, \Fbar}$ denotes its
base change to $\Fbar$.
First recall \emph{loc.\ cit.} that, given an algebraic group $\G_1$
over $F$ and a highest-degree invariant differential
form $\omega_1$ on $\G_{1, \Fbar}$, we can
attach a Haar measure $|\omega_1|$ on $\G_1(F)$
by choosing any $\mu_1 \in \Fbar^{\times}$ such
that $\mu_1 \omega_1$ is defined over $F$
(such a $\mu_1$
exists by Hilbert's Theorem 90), and setting
$|\omega_1| := |\mu_1|^{-1} |\mu_1 \omega_1|$.
One similarly obtains a Haar measure on $\G_1(E')$
for every finite extension $E'/F$ (using $\omega_1$ and the normalized
absolute value $|\cdot|_{E'}$ on $E'$),
and this measure will be denoted by $|\omega_1|_{E'}$.
Choose highest-degree differential forms $\omega_{\G}$, $\omega_{\H}$
and $\omega_{\T}$ on $\G$, $\H$ and $\T$ respectively, and
use these to fix Haar measures $dg$,
$dh$ and $dt$ on $\G(F)$, $\H(F)$ and $\T(F)$
respectively, in the manner just described.
By transport of structure from $\T$ via inner automorphisms,
we can choose highest-degree forms $\omega_{\T'}$ on 
each maximal torus $\T'$ of $\G_{\Fbar}$ (well defined up to
scaling by $\mO_{\Fbar}^{\times}$). Further, the endoscopic
datum also allows us to transfer $\omega_{\T}$ to a highest-degree
differential form $\omega_{\T''}$ on each maximal torus
$\T''$ of $\H_{\Fbar}$ (see \cite[\Sec1.4]{LS87}).
For a maximal torus $\T'$ of $\G$ or $\H$
defined over $F$, we therefore get an associated measure
$dt' = |\omega_{\T'}|$ on $\T'(F)$,
which (unlike $\omega_{\T'}$) does not depend on any choice other than,
of course, that of $\omega_{\T}$.

\begin{remark}
From now on, until \Sec\ref{adjoint quotients map}, we will
state and prove certain results for $\G$. Since $\H$ satisfies the same
hypotheses as $\G$, we may and shall later apply them in the context of
$\H$, too.
\end{remark}

Recall that for each $y \in \mcB(\G)$ and $r \in \RR$, we also have a
Moy--Prasad lattice $\g_{y, r}^* \subset \g^*(F)$ given by:
\begin{equation} \label{gyr* definition}
\g_{y, r}^* =
\set{\Upsilon \in \g^*(F) }{ \Upsilon(\g_{y,(-r)+}) \subseteq \varpi \mO_F }.
\end{equation}
By Proposition 4.1 of \cite{AR00},
thanks to $p$ being \ggood/,
there exists an $\Ad \G$-invariant symmetric nondegenerate bilinear form
$\langle \cdot , \cdot \rangle$ on $\g$
that induces an identification of $\g_{y, r}$ with $\g_{y, r}^*$
for each $y \in \mcB(\G)$ and $r \in \RR$. Fix one such.

\begin{remark} \label{AR00 bilinear form}
In \cite{AR00} such a form $\langle \cdot, \cdot \rangle$ is constructed
as the restriction to $\g(F)$ of a bilinear form on $\g(E_1)$
satisfying analogous properties over $E_1$, where $E_1/F$ is
an extension that splits $\G$. Thus, we may and do assume that for each
$y \in \mcB(\G, E)$,
$\langle \cdot, \cdot \rangle$ induces an identification
of the lattice
$\g(E)_{y, r} \subset \g(E)$ with the lattice
$\g^*(E)_{y, r} \subset \g^*(E)$.
In particular, if $y$
is hyperspecial over $E$, $\langle \cdot, \cdot \rangle$
induces a perfect pairing on $\g(E)_{y, 0}$.
\end{remark}

Now let $X \in \g(\Fbar)$ be regular semisimple, so
that its centralizer $\T_X \subset \G_{\Fbar}$
is a maximal torus.
The Lie algebra $\t_X$ of $\T_X$ is also the
kernel of $\ad X$.
We now have two 
$\G_{\Fbar}$-invariant
top-degree differential forms  (which are well defined
modulo $\mO_{\Fbar}^{\times}$-scaling) on the $\G_{\Fbar}$-orbit
of $X$,
which may and shall be identified with
$\G_{\Fbar}/\T_X$. The first is $\omega_X :=
\omega_{\G}/\omega_{\T_X}$, and the second is the differential
form $\omega_X'$ arising from the nondegenerate symplectic
form $\langle \cdot, \cdot \rangle_X$ on the tangent space
$\g_{\Fbar}/\t_{\X}$ to
the variety $\G_{\Fbar}/\T_X$ at $1 \cdot \T_X$ induced
by the degenerate symplectic pairing on $\g_{\Fbar}$
defined by $\langle v, w \rangle_X := \langle X, [v, w] \rangle
= \langle [X, v], w \rangle$.
Thus, if $e_1, \dots, e_{2r}$ is an ordered symplectic basis for
$\langle \cdot, \cdot \rangle_X$, i.e., $\langle e_i, e_j \rangle_X
= \delta_{i(2r + 1 - j)}$ for $1 \leq i \leq r$, then $\omega_X'$
takes the value one on $e_1 \wedge \dots \wedge e_{2r}$.
If $X$ and each $e_i$ are defined over $F$ and the $\mO$-lattice spanned
by the $e_i$ in $\g(F)/\t_X(F)$ is the image of an $\mO$-lattice $L$
in $\g(F)$, then the measure $|\omega_X'|$
on $\G(F)/\T_X(F)$ is the quotient of the measures on
$\G(F)$ and $\T_X(F)$ corresponding to the measures on $\g(F)$ and $\t_X(F)$
normalized by $L$ and $L \cap \t_X(F)$, respectively.

\begin{remark} \label{under this notation}
For any finite extension $E'/F$ in $\Fbar$,
given any differential form
$\omega$ on $\g\times_F \Fbar$,
the measure
$|\omega|_{E'}$ on $\g(E')$ can be described as follows.
Given a lattice $L \subset \g(E')$,
let $a$ denote a generator of
the lattice determined by $L$ inside the top exterior power
$\bigwedge^{\dim \g}_{E'} \g(E')$.  For example, $a$ could be $e_1 \wedge
\dots \wedge e_{\dim \g}$ for any $\mO_{E'}$-basis $e_1, \dots, e_{\dim \g}$
of $L$.
Let $\tilde m(L, \omega) \in {\Fbar}^{\times}/\mO_{\Fbar}^{\times}$
denote the image of $\omega(a)$.
This is independent of our choice of $a$.
Then
\[
\meas(L, |\omega|_{E'}) = |\tilde m(L, \omega)|_{E'}
=
|\tilde m(L, \omega)|_F^{f_{E'/F}},
\]
where $f_{E'/F}$ denotes the residue degree of $E'/F$.
We define
$$
\mm(L, |\omega|) := \meas(L, |\omega|_{E'})^{1/f_{E'/F}}.
$$
Note that this definition has the following invariance property:
if $E', L, \omega$ are as above, and $E'' \supset E'$ is a finite extension
inside $\Fbar$, then $\mm(L, |\omega|) = \mm(L \otimes_{\mO_{E'}}
\mO_{E''}, |\omega|)$.
\end{remark}

\subsection{The relation between $|\omega_X|$ and $|\omega_X'|$} \label{omegas relation}

\begin{lm} \label{omegaX omegaX'}
The measures $|\omega_X|$ and $|\omega_X'|$ are related by
\[
|\omega_X'|
=
D_{\g}(X)^{1/2}
\dfrac{\mm(\t(\mO_E), |\omega_{\T}|)}{\mm(\g(E)_{x, 0}, |\omega_{\G}|)}
,
\]
where $D_{\g}(X) = |\det(\ad X; \g/\t_X)|$.
\end{lm}
\begin{proof}
Using Remark \ref{under this notation} and that $|a \omega| =
|a| |\omega|$ for any top-degree form $\omega$ on $\G_{\Fbar}$
and $a \in \Fbar^{\times}$, one sees that
it suffices to prove:
\[
\omega_X' \in \mO_{\Fbar}^{\times} \cdot
\det(\ad X; \g/\t_X)^{1/2}
\dfrac{\tilde m(\t(\mO_E), \omega_{\T})}{\tilde m(\g(E)_{x, 0}, \omega_{\G})}
\cdot \omega_X
\]
(the choice of square-root clearly does not matter).
This formulation allows us to change the base field:
since the assignments $X \mapsto \omega_X$ and $X \mapsto \omega_X'$
behave well with respect to $\G(\Fbar)$-conjugation (in the sense
that $\omega_X \in \omega_{\Ad g(X)} \cdot \mO_{\Fbar}^{\times}$
and similarly for $\omega_X'$), we may and do assume
that $X \in \t(E')$ for some finite extension $E'$ of $E$ contained in $\Fbar$.
Thus, $\T_{E'} = \T_{X, E'}$. 
For each root $\alpha\in R(\B,\T)$,
by
the construction of the form 
$\langle \cdot , \cdot \rangle$,
there exists $a_{\alpha} \in \mO_E^{\times}$ such that
$\langle a_{\alpha} X_{\alpha}, X_{- \alpha} \rangle = 1$, where
for this proof, we write $X_{- \alpha}$
to mean a fixed basis element for the root space $\g_{- \alpha}(\mO_E)$
(recall that our fixed pinning realizes $\G$ as a split
reductive group over $\mO_E$).
Then an ordered symplectic
basis $e_1, \dots, e_{\dim \G - \rk \G}$ for $\langle \cdot, \cdot \rangle_X$ over $E'$ may be
chosen to have as its underlying set:
\[
\set{a_{\alpha} \cdot d\alpha(X)^{-1} X_{\alpha}}{\alpha > 0}
\cup
\set{X_{- \alpha}}{\alpha > 0},
\]
with $d\alpha$ denoting the derivative of $\alpha$.
The form $\omega_X'$ then takes the value $1$ at
$e_1 \wedge \dots \wedge e_{\dim \G - \rk \G}$,
whereas, the image of $\omega_X(e_1 \wedge \dots \wedge
e_{\dim \G - \rk \G})$ in ${\Fbar}^{\times}/\mO_{\Fbar}^{\times}$
equals
\[
\det(\ad X; \g/\t_X)^{-1/2}
\dfrac%
{\tilde m(\g(E)_{x, 0} \otimes_{\mO_E} \mO_{E'}, \omega_{\G})}
{\tilde m(\t(\mO_{E'}), \omega_{\T})}
=
\det(\ad X; \g/\t_X)^{-1/2}
\dfrac%
{\tilde m(\g(E)_{x, 0}, \omega_{\G})}
{\tilde m(\t(\mO_E), \omega_{\T})}
\]
(use that 
$\prod_{\alpha > 0} d\alpha(X) \in
\mO_{\Fbar}^{\times} \cdot \det(\ad X; \g/\t_X)^{1/2}$, and also
use Remark \ref{under this notation} again), proving the lemma.
\end{proof}

\subsection{Studying $\omega_X'$ for semisimple $X \in Y + \g_{x, 0+}$} \label{studying omegaX'}
\begin{notn}
If $L$ is a lattice (of full rank) in a finite-dimensional vector
space $V$ over $F$ with a symmetric
or alternating bilinear form $B$, then we set
(suppressing dependence on $B$ for lightness of notation)
\[
L^{\perp} = \set{v \in V}{ B(v, L) \subseteq \mO}.
\]
\end{notn}

\begin{lm} \label{selfdual}
Suppose $X \in Y + \g_{x, 0+}$ is semisimple.
Then, in the symplectic
space $(\g(F)/\t_X(F), \langle \cdot, \cdot \rangle_X)$,
\begin{equation} \label{dual lattice}
\left( \g_{x, 0+}/ (\g_{x, 0+} \cap \t_X(F)) \right)^{\perp} =
\g_{x, -1}/ (\g_{x, -1} \cap \t_X(F)).
\end{equation}
\end{lm}

\begin{proof}
By Equation \eqref{gyr* definition} and the choice of
$\langle \cdot, \cdot \rangle$, we have $\g_{x, 0+}^{\perp}
= \varpi^{-1} \g_{x, 0} = \g_{x, -1}$, with $\perp$ being taken
with respect to the bilinear form $\langle \cdot, \cdot \rangle$.
This together with the fact that $\g_{x, 0+}$ is invariant under
$\ad X$ gives the relation `$\supseteq$'.

For the reverse inclusion, it is enough to show that any element
$Z \in \g(F)$ such that $[X, Z] \in \g_{x, -1}^{\perp} = \g_{x, 0+}$
($\perp$ taken with respect to $\langle \cdot, \cdot \rangle$)
necessarily belongs to
$\t_X(F) + \g_{x, 0+} = \ker \ad X + \g_{x, 0+}$.
In other words,
it is enough to show that:
\begin{equation} \label{Xgx0p}
[X, \g_{x, 0+}] = [X, \g(F)] \cap \g_{x, 0+}.
\end{equation}
This is equivalent to showing
that the rank of the endomorphism $\overline{\ad X}$ of
$\g_{x, 0+}/\varpi \g_{x, 0+}$ induced by $\ad X$, which a priori
is at most $\dim_F [X, \g(F)] = \dim \G - \rk \G$, is actually equal to
$\dim \G - \rk \G$. Now 
$\overline{\ad X}$  preserves the Moy--Prasad filtration of
$\g_{x, 0+}/\varpi \g_{x, 0+} = \g_{x, 0+}/\g_{x, 1+}$ induced by
the lattices $\g_{x, r+}$ (for $0 < r < 1$),
and we can consider the associated
graded map $\overline{\ad X}_{\mathrm{gr}}$.
A priori, the
rank of $\overline{\ad X}_{\mathrm{gr}}$ is at most that of
$\overline{\ad X}$, so it suffices to show that
the rank of $\overline{\ad X}_{\mathrm{gr}}$ equals
$\dim_F [X, \g(F)] = \dim \G - \rk \G$.
The map $\overline{\ad X}_{\mathrm{gr}}$
is the same as
the associated graded map of the analogously defined endomorphism
$\overline{\ad Y}$. Since $p$ is \gFgood/ by Remark \ref{p is gFgood},
Corollary \ref{extending from tamely ramified}(iii) applies to show,
using notation from there, that
the codimension of the image of $\overline{\ad Y}$ on
$\g_{x, r}/\g_{x, r+}$ equals $\dim_{\kappa} L_r/L_{r+}$. Thus, the rank
of $\overline{\ad X}_{\mathrm{gr}}$ equals $\dim \g -
\dim_{\kappa} (L_{0+}/\varpi L_{0+})$, which equals $\dim \g -
\rk \g$.
\end{proof}

\subsection{Map of adjoint quotients} \label{adjoint quotients map}
Recall that we are assuming that $\H$ splits over $E$.
Just as we did with $\G$ in \Sec\ref{set up}, we fix a $\Gal(\Fbar/F)$-stable
pinning for $\H$ over $E$ and get a point
$x_{\H}$ in the reduced Bruhat--Tits building $\mcB(\H)$ of $\H$
as well as a regular nilpotent element $Y_{\H} \in \h(F)$.

Let $\chi_{\H} : \h \longrightarrow \q_{\H}$ denote the
adjoint quotient of $\h$
(as in Notation \ref{adjoint quotient notation}).
Following the notational set up of \cite[\Sec1.1.6]{KV12},
the transfer of stable conjugacy
classes from $\h$ to $\g$ is described by a finite morphism
$\nu : \q_{\H} \longrightarrow \q$.
Namely, for each maximal torus $\T_{\H}$ of $\H$, there exists a certain stable
conjugacy class of embeddings
$\iota : \T_{\H} \hookrightarrow \G$ each of which defines an isomorphism
of $\T_{\H}$ onto some maximal torus  $\T'$
of $\G$ (any such $\iota$ is an admissible embedding as named
in \cite[\Sec1.3]{LS87}),
and all of which satisfy that $\chi \circ d\iota = \nu \circ \chi_{\H}$.

Recall that functions $\phi \in C_c^{\infty}(\g(F))$ and $\phi_{\H}
\in C_c^{\infty}(\h(F))$ are said to have matching orbital integrals
if and only if for all $\G$-regular semisimple $X_{\H} \in \h(F)$
(i.e., $\nu \circ \chi_{\H}(X_{\H}) = \chi(X')$ for some regular
semisimple $X' \in \g(\Fbar)$), we have an equality
\begin{equation} \label{matching orbital integrals}
\dsum_{X_{\H}'}
I \left(X_{\H}', \phi_{\H}, \bigl| \omega_{\H}/\omega_{\T_{X_{\H}'}} \bigr| \right)
= \dsum_{X} \Delta_0'(X_{\H}, X)
\,
I
\negthinspace
\left(X, \phi, \left| \omega_{\G}/\omega_{\T_X} \right| \right),
\end{equation}
where we have written $\T_Z$ for the centralizer of $Z$ in the appropriate
group, $I$ stands for \emph{normalized} orbital integral, $\Delta_0'$ denotes
the transfer factor that excludes the term $\Delta_{\mathrm{IV}}$
(which is accounted for by the normalization of orbital integrals),
$X_{\H}'$
runs over a set of representatives for the $\H(F)$-conjugacy classes in
the stable conjugacy class of $X_{\H}$ (this stable conjugacy class equals
$\chi_{\H}|_{\H(F)}^{-1}(\chi_{\H}(X_{\H}))$),
and $X$ runs over a set of representatives
for the $\G(F)$-conjugacy classes in
$\chi|_{\G(F)}^{-1} \circ \nu \circ \chi_{\H}(X_{\H})$.
Here the transfer factors need to be normalized, and we do so following
\cite{Kot99}, namely, normalizing them according to the $F$-conjugacy
class of the pinning $(\B^-, \T, \{X_{- \alpha}\})$,
where the vectors $X_{- \alpha}$ are
normalized to satisfy
$[X_{\alpha}, X_{- \alpha}] = H_{\alpha} = d\alpha^{\vee}(1)$.

\begin{remark} \label{on transfer factor normalization}
Let us state the consequence of this normalization of transfer factors that
concerns us.
Kottwitz \cite[p.\ 128]{Kot99}
associates a regular
nilpotent element to the pinning $(\B^-, \T, \{X_{- \alpha}\})$.
For us, this element equals $Y$.
Corollary \ref{extending from tamely ramified}
gave us a Kostant section $Y + L_F$.
Then \cite[Theorem 5.1]{Kot99} says that $\Delta_0'(X_{\H}, X)
= 1$ whenever $X_{\H} \in \h(F)$ and $X \in Y + L_F \subset \g(F)$ are
regular semisimple elements that match in the sense that
$\nu \circ \chi_{\H}(X_{\H}) = \chi(X)$.
By Lemma \ref{AD04 for unramified}(i),
the previous sentence holds true with $Y + L_F$ replaced by
$Y + \g_{x, 0+}$.
\end{remark}

\subsection{Some consequences of a Kazhdan--Varshavsky quasi-logarithm}
\label{KV consequences}
To work with orbital integrals, we will need to
relate the measure of $\T_X(F) \cap \G_{x, r}$ to that
of $\t_X(F) \cap \g_{x, r}$. For this purpose alone, we will use
a Kazhdan--Varshavsky quasi-logarithm.

\begin{remark} \label{Exponential map}
Identify $\mcB(\G)$ with the Bruhat--Tits building
of $\G_{\scn}$ as well. Then
\begin{enumerate}[(i)]
\item
Since $\G$ satisfies Hypothesis \ref{hyp:Sec3}(\ref{item:isogeny}),
it follows from \cite[Lemma 8.12]{BKV16} that the
obvious isogeny $\Z^0 \times \G_{\scn} \longrightarrow \G$ induces
an isomorphism
\[
\Z^0_{0+} \times \G_{\scn, x, 0+}
\overset{\sim}{\longrightarrow}
\G_{x, 0+},
\]
of $p$-adic analytic groups,
where $\Z^0_{0+}$ stands for $\Z^0_{z, 0+}$,
$z$ denoting the
unique point in the reduced Bruhat--Tits building of $\Z^0$.
\item
By Lemmas C.3
and C.4 of \cite{BKV16}, we have an analytic isomorphism
$\mcL : \G_{\scn, x, 0+} \longrightarrow \g_{\scn, x, 0+}$ that is
equivariant under $\G_{\scn, x, 0}$-conjugation, and such
that for all $r > 0$, $\mcL(\G_{\scn, x, r}) = \g_{\scn, x, r}$.
\end{enumerate}
\end{remark}

The consequences of Remark \ref{Exponential map} that we wish to use
are collected in the following corollary:
\begin{cor} \label{exponential map consequences}
Let $X \in \g(F)$ be regular semisimple with
centralizer $\T_{\scn, X} \subset \G_{\scn}$.
\begin{enumerate}[(i)]
\item $\mcL$ takes
$\T_{\scn, X}(F) \cap \G_{\scn, x, r}$ homeomorphically
onto $\t_{\scn, X}(F) \cap \g_{\scn, x, r}$, for all $r > 0$.

\item Suppose that
$\g_{\scn}(F)$ and $\G_{\scn}(F)$ are given compatible
measures, and that so are
$\t_{\scn, X}(F)$ and $\T_{\scn, X}(F)$. Then, for $r > 0$,
$\meas(\g_{\scn, x, r}) = \meas(\G_{\scn, x, r})$, and
$\meas(\t_{\scn, X}(F) \cap \g_{\scn, x, r}) =
\meas(\T_{\scn, X}(F) \cap \G_{\scn, x, r})$.
\end{enumerate}
\end{cor}
\begin{proof}
It is enough to see statements (i) and (ii) above when $\G$ is simply connected.
For all $r > 0$, by Remark \ref{Exponential map}(ii), we have homeomorphisms $\G_{x, r} \rightarrow
\g_{x, r}$ and $\G_{x,0+} \rightarrow \g_{x,0+}$.
Now
(i) follows from the conjugation equivariance of $\mcL$: picking any
$t \in \T_X(F)$ that is strongly regular in $\G$ (e.g., $\exp(a X)$ with
$a \in F$, $|a|$ small enough), we see that any given $g \in \G_{x,0+}$
belongs to the kernel $\T_X$ of $\Int t$ if and only if $\mcL(g) \in
\g_{x,0+}$ belongs to the kernel $\t_X$ of $\Ad t$.

To see that (ii) follows too, it is enough to show that the top
exterior power of the derivative of $\mcL$ (resp., $\mcL|_{\T_X(F)}$)
at each $g \in \G_{x, r}$ (resp., at each $t \in \T_X(F) \cap \G_{x, r}$),
which is an endomorphism of a one-dimensional $F$-vector space,
is a unit in $\mO$, or equivalently, $\mO_{\Fbar}$. Let $E_X \subset \Fbar$
be a field extension over which $\T_X$ splits. Now $\G_{E_X}$
gets an integral model $\G_{\mO_{E_X}}$ from $x$,
while $\T_{X, E_X}$
has a canonical integral model since it is split.
We claim that the base changes
of $\mcL$ and $\mcL|_{\T_X}$ to $E_X$ extend to $\mO_{E_X}$-morphisms $\G_{\mO_{E_X}}
\rightarrow \g_{\mO_{E_X}}$ and $\T_{X, \mO_{E_X}} \rightarrow
\t_{X, \mO_{E_X}}$. Indeed, the assertion involving $\G_{\mO_{E_X}}$
follows from $\mcL$ being `defined over $\mO$'
in the sense of \cite[Appendix C.2]{BKV16}
(see the last sentence of \cite[C.2(a)]{BKV16},
\cite[Definition 1.88(b) and Notation 1.8.6(b)]{KV06}). This
together with the conjugation equivariance of $\mcL$ implies,
as in (i), that $\mcL$ takes $\T_{\mO_{E_X}}(\mO_{\Fbar})$ to
$\t_{\mO_{E_X}}(\mO_{\Fbar})$, from which the assertion involving
$\T_{X, \mO_{E_X}}$ follows by \cite[Proposition 1.7.6]{BT2}.

This already implies that the top
exterior power of the derivative of $\mcL$ (resp., $\mcL|_{\T_X(F)}$)
at each $g \in \G_{x, r}$ (resp., at each $t \in \T_X(F) \cap \G_{x, r}$)
belongs to $\mO_{\Fbar}$, and it suffices to show that the image
of this element in $\bar \kappa$ equals $1$. But this image may be
computed by base-changing 
to $\kappa_{E_X}$. However, $g$ (resp., $t$) has the identity
for its image in $\G_{\mO_{E_X}}(\kappa_{E_X})$ (resp.,
$\T_{\mO_{E_X}}(\kappa_{E_X})$), so that the result follows from
the derivative of $\mcL$ at the identity element being identity,
by the very definition of a quasi-logarithm (see
\cite[C.1(a)]{BKV16}).
\end{proof}

\subsection{An orbital integral computation} \label{orbital integral computation}

\begin{lm} \label{orbital integral}
Suppose $X \in \g(F)$ is regular semisimple. Then
\begin{equation} \label{orbital integral equation}
I(X, \charfn_{Y + \g_{x, 0+}}, |\omega_X|)
=
\begin{cases}
c_{\G} & \text{if $X \in \Ad \G(F) (Y + \g_{x, 0+})$, and} \\
0 & \text{otherwise},
\end{cases}
\end{equation}
where

\begin{eqnarray*}
c_{\G}
&=& \mm(\t(\mO_E), |\omega_{\T}|)^{-1} \mm(\g(E)_{x, 0}, |\omega_{\G}|)
q^{-(\dim \G - \rk \G + m)/2},
\end{eqnarray*}
where $m$ is the rank of the endomorphism of the $\kappa$-vector
space $\g_{x, 0}/\g_{x, 0+}$ induced by $\ad Y$.
\end{lm}

\begin{proof}
We may and do assume $X \in Y + \g_{x, 0+}$.
By Lemma \ref{omegaX omegaX'} the left-hand side is 
given by the product of
$\mm(\t(\mO_E), |\omega_{\T}|)^{-1} \mm(\g(E)_{x, 0}, |\omega_{\G}|)$
and the unnormalized orbital integral
$O(X, \charfn_{Y + \g_{x, 0+}}, |\omega_X'|)$.
Hence by Proposition \ref{AD04 main result},
the left-hand side equals
\[
\mm(\t(\mO_E), |\omega_{\T}|)^{-1} \mm(\g(E)_{x, 0}, |\omega_{\G}|)
\cdot
\meas(\G_{x, 0+}/(\G_{x, 0+} \cap \T_X(F)), |\omega_X'|).
\]
Note that thanks to Hypothesis \ref{hyp:Sec3}(\ref{item:isogeny}),
the projection $X_{\scn}$ of $X$ to $\g_{\scn}(F) \subset \g(F)$
belongs to $Y + \g_{\scn, x, 0+}$, and
\begin{multline*}
\meas(\G_{x, 0+}/(\G_{x, 0+} \cap \T_X(F)), |\omega_X'|)
= \meas(\g_{\scn, x, 0+}/(\g_{\scn, x, 0+} \cap \t_{X_{\scn}}(F)), |\omega_{X_{\scn}}'|) \\
= [\g_{\scn, x, -1}/(\g_{\scn, x, -1} \cap \t_X(F)) : \g_{\scn, x, 0+}/(\g_{\scn, x, 0+} \cap \t_X(F))]^{-1/2}
\end{multline*}
(use Remark \ref{Exponential map}(i),
Corollary \ref{exponential map consequences}(ii)
and Lemma \ref{selfdual}).
Since the index of $\g_{\scn, x, 0}/(\g_{\scn, x, 0} \cap \t_X(F))$
in $\g_{\scn, x, -1}/(\g_{\scn, x, -1} \cap \t_X(F))$ 
equals $q^{\dim \G - \rk \G}$,
it now suffices to show that
\begin{equation} \label{relating to Y}
\dim_{\kappa} \g_{x, 0}/ ((\g_{x, 0} \cap \t_X(F)) + \g_{x, 0+})
= \dim_{\kappa} ([Y, \g_{x, 0}] + \g_{x, 0+})/\g_{x, 0+}.
\end{equation}
Thanks to Equation \eqref{Xgx0p},
the left-hand side of the above equation is
the rank of the $\kappa$-linear endomorphism of
$\g_{x, 0}/\g_{x, 0+}$ induced by $\ad\, X$.
Now the lemma follows
from the fact that $\ad\ X$ and $\ad\ Y$
induce the same map on $\g_{x, 0}/\g_{x, 0+}$, since
$X \in Y + \g_{x, 0+}$.
\end{proof}

\begin{remark}
Note that if $p$ is large enough to satisfy the hypotheses
of \cite[\Sec4.2]{Deb02a}
(or equivalently, those of \cite[\Sec 2.2]{Deb02b})
one might also be able to prove Lemma \ref{orbital integral} as follows.
Under this assumption, 
the main results
of \cite{AD04} and 
\cite[Theorem 2.1.5]{Deb02b}
imply that every $X \in \g(F)_{\tn}$
lies in the range of validity of the Shalika germ expansion for
the function $\charfn_{Y + \g_{x, 0+}}$ near $0 \in \g(F)$.
This in turn implies 
Lemma \ref{orbital integral}
up to some constant independent of $X$.
It should be possible
to calculate this constant using the main result of \cite{She89}
and comparing two different measures on $\Ad \G(F) \cdot Y$:
the one considered by \cite{She89} and the one considered just before
\cite[Lemma 3.4.4]{Deb02b}.
\end{remark}

\subsection{The main result} \label{main result 2}
Recall we have assumed 
that $\G$ and $\H$ satisfy
Hypothesis \ref{hyp:Sec3}.
\begin{pro} \label{x k transfer}
Let $\phi_{\H} := \charfn_{Y_{\H} + \h_{x_{\H}, 0+}}$
and $\phi := \charfn_{Y + \g_{x, 0+}}$. Let $c_{\G}$
be as in Lemma \ref{orbital integral}, and let
$c_{\H}$ be the analogous constant for $\H$.
Then $c_{\H}^{-1} \phi_{\H}$ and
$c_{\G}^{-1} \phi$ have matching orbital integrals.
\end{pro}
\begin{proof}
Let $X_{\H} \in \h(F)$ be $\G$-regular semisimple.
We need to prove an equality analogous to that in
Equation \eqref{matching orbital integrals}, which we may write as:
\begin{equation} \label{new matching orbital integrals}
\dsum_{X_{\H}'}
I
\left(
	X_{\H}',
	c_{\H}^{-1} \phi_{\H},
	\bigl|
		\omega_{\H}/\omega_{\T_{X_{\H}'}}
	\bigr|
\right)
=
\dsum_{X} \Delta_0'(X_{\H}, X)
\,
I
\negthinspace
\left(
	X,
	c_{\G}^{-1} \phi,
	\left|
		\omega_{\G}/\omega_{\T_X}
	\right|
\right).
\end{equation}
Since $\G$ is quasi-split, it is known that the sum on
the right-hand side of Equation \eqref{new matching orbital integrals}
is nonempty
(e.g., this is
an easy consequence of either of
\cite[Corollary 2.2]{Ko82} or \cite[Theorem 4.1]{Ko82}).
It is also easy to see
(see, e.g., \cite[\Sec 1.1.7]{KV12}) that for each $X$ contributing
to the right-hand side of that equation, one
can choose an isomorphism $\iota : \T_{\H, X_{\H}} \longrightarrow
\T_X$ of the centralizers of $X_{\H}$ and $X$ with $d \iota(X_{\H}) = X$.
By Definition \ref{top nilp}, we conclude
\begin{equation} \label{topological nilpotence and matching}
\text{$X_{\H} \in \h(F)_{\tn}$ if and only if $X \in \g(F)_{\tn}$.}
\end{equation}

If $X_{\H} \not\in \h(F)_{\tn}$, then by
(\ref{topological nilpotence and matching}), $X \not\in \g(F)_{\tn}$
for any $X$ that contributes to the right-hand side of Equation
\eqref{new matching orbital integrals}.
Recalling that $\Ad \G(F)(Y + \g_{x, 0+}) \subset \g(F)_{\tn}$
by Proposition \ref{AD04 main result}(i), and
similarly for $\H$,
both sides of Equation \eqref{new matching orbital integrals} vanish when
$X_{\H} \not\in \h(F)_{\tn}$.

Thus, now assume $X_{\H} \in \h(F)_{\tn}$, so that, by
(\ref{topological nilpotence and matching}), any
$X$ that contributes to the right-hand side of Equation
\eqref{new matching orbital integrals} belongs to $\g(F)_{\tn}$.
By Proposition \ref{AD04 main result}(ii) applied to $\H$, exactly one
$X_{\H}'$ contributes to the left-hand side of
Equation \eqref{new matching orbital integrals}, say $X_{\H}$, which
may be assumed to lie in $Y_{\H} + \h_{x_{\H}, 0+}$. By the
same result but applied to $\G$, exactly one
$X$ contributes to the right-hand side of
Equation \eqref{new matching orbital integrals},
and such an $X$ may be assumed to lie
in $Y + \g_{x, 0+}$. Thus, we now need to
show that
\[
I
\left(
	X_{\H},
	c_{\H}^{-1} \phi_{\H},
	\bigl|
		\omega_{\H}/\omega_{\T_{X_{\H}}}
	\bigr|
\right)
=
\Delta_0'(X_{\H}, X)
\,
I
\negthinspace
\left(
	X,
	c_{\G}^{-1} \phi,
	\left| \omega/\omega_{\T_X} \right|
\right).
\]
By Remark \ref{on transfer factor normalization},
$\Delta_0'(X_{\H}, X) = 1$. Now the result follows
from Lemma \ref{orbital integral}.
\end{proof}

\subsection{Some more pairs of functions with matching orbital integrals}
\label{more pairs of matching functions}
We expect that one
can use the above result to produce more pairs of functions
with matching orbital integrals. Let us sketch how this may
be done.
\begin{enumerate}[(a)]
\item
A `scaling' argument as in
\cite[Prop.\ 3.2.2]{Fer07}
(similar to an argument in \cite[\Sec9]{Sha90})
shows that for all $l \in \ZZ$,
up to a scalar (depending on $l$)
the characteristic functions of $\varpi^{-l} Y + \g_{x, (-l)+}$ and
$\varpi^{-l} Y_{\H} + \h_{x_{\H}, (-l)+}$ have matching orbital integrals.
\item Let $l \in \NN$. Recall the bilinear
form $\langle \cdot, \cdot \rangle$ on $\fg$.
It is explained in \cite[\Sec VIII.6]{Wal95} how to obtain from
$\langle \cdot, \cdot \rangle$
a symmetric nondegenerate $\Ad \H$-invariant
bilinear form $\langle \cdot, \cdot \rangle_{\H}$ on $\h$.
Using $\langle \cdot, \cdot \rangle$,
$\langle \cdot, \cdot \rangle_{\H}$
and an additive character $\Lambda$ of $F$ with conductor $\varpi \mO_F$ to
define Fourier transforms on $\g(F)$ and $\h(F)$, we get
two functions $\varphi_l \in C_c^{\infty}(\g(F))$ and $\varphi_{\H, l} \in
C_c^{\infty}(\h(F))$ that have matching orbital integrals, up to an
explicit scalar (use Conjecture 1 of \cite{Wal95}, which has since been
proved thanks to Waldspurger and Ng\^{o},
see \cite[Theorem 4.1.3]{KV12}). Here $\varphi_l$ is supported on $\g_{x, l}$,
on which it
is a scalar multiple of
$X \mapsto \Lambda(\langle \varpi^{-l} Y, X \rangle)$, and
$\varphi_{\H, l}$ has a similar description.
\item
If we have suitable quasi-logarithms on $\H$ and $\G$
that are compatible with each other, and behave well with respect to
transfer factors, one should be able to pull the functions in (b) above to the
group level. Thus, if suitable hypotheses on the quasi-logarithm maps
are satisfied (basically, statements along the lines of
Proposition 5.1.3 and Proposition 5.2.5(b) and (c) of \cite{KV12}),	       
we should get that for
$l \in \NN$, $\tilde \varphi_l \in C_c^{\infty}(\G(F))$ and
$\tilde \varphi_{\H, l} \in C_c^{\infty}(\H(F))$
have matching orbital integrals, up to a scalar.
Here $\tilde \varphi_l$ is supported on $\G_{x, l}$, on which it is the
inflation of the character
of $\G_{x, l}/\G_{x, l+}$ obtained
by composing the isomorphism $\G_{x, l}/\G_{x, l+} \cong
\g_{x, l}/\g_{x, l+}$ of groups with (a scalar multiple of)
$X \mapsto \Lambda ( \langle \varpi^{-l} Y, X \rangle)$.

\end{enumerate}


\def\cprime{$'$} \def\cprime{$'$} \def\Dbar{\leavevmode\lower.6ex\hbox to
  0pt{\hskip-.23ex \accent"16\hss}D}

\end{document}